\documentclass[oneside,reqno]{amsart}
\usepackage{amsmath,amssymb}\allowdisplaybreaks
\usepackage{xcolor}
\usepackage{mathdots}
\usepackage[a4paper]{geometry}
\usepackage{graphicx}
\usepackage{bookmark}
\usepackage{hyperref}
\hypersetup{
    colorlinks=true,
    linkcolor=blue,
    citecolor=red,
    filecolor=magenta,      
    urlcolor=cyan,
}

\newtheorem{theorem}{Theorem}[section]
\newtheorem{lemma}[theorem]{Lemma}
\newtheorem{proposition}[theorem]{Proposition}
\newtheorem{remark}[theorem]{Remark}
\newtheorem{corollary}[theorem]{Corollary}

\theoremstyle{definition}
\newtheorem*{definition}{Definition}

\title{Hankel determinants of a Sturmian sequence}
\author{Haocong Song}
\address[H.-C. Song]{School of Mathematics, South China University of Technology, Guangzhou 510640, China}
\email{song96925@gmail.com}
\author{Wen Wu$^*$}\thanks{$^*$Wen Wu is the corresponding author.}
\address[W. Wu]{School of Mathematics, South China University of Technology, Guangzhou 510640, China}
\email[corresponding author]{wuwen@scut.edu.cn}
\date{}

\begin{document}
	\begin{abstract}
		Let $\tau$ be the substitution $1\to 101$ and $0\to 1$ on the alphabet $\{0,1\}$. The fixed point of $\tau$ leading by 1, denoted by $\mathbf{s}$, is a Sturmian sequence. We first give a characterization of $\mathbf{s}$ using $f$-representation. Then we show that the distribution of zeros in the determinants induces a partition of integer lattices in the first quadrant. Combining those properties, we give the explicit values of the Hankel determinants $H_{m,n}$ of $\mathbf{s}$ for all $m\ge 0$ and $n\ge 1$. 
	\end{abstract}
	\maketitle
	\setcounter{tocdepth}{1}
	\tableofcontents

	\section{Introduction}
	Let $\mathbf{s}=(s_j)_{j\ge 0}$ be an integer sequence. For all  $m\geq 0$, $n\geq 1$, the $(m,n)$-order Hankel matrix of $\mathbf{s}$ is  
    \[M_{m,n}:=(s_{m+i+j})_{0\leq i,j\leq n-1}=
    \begin{pmatrix}
        s_m & s_{m+1} & \cdots & s_{m+n-1}\\
        s_{m+1} & s_{m+2} & \cdots & s_{m+n}\\
        \vdots & \vdots & \ddots & \vdots\\
        s_{m+n-1} & s_{m+n}  & \cdots & s_{m+2n-2}.
    \end{pmatrix}.\]
    The $(m,n)$-order Hankel determinant of $\mathbf{s}$ is $H_{m,n}=\det M_{m,n}$.

	Hankel determinants of automatic sequences have been widely studied, due to its application to the study of irrationality exponent of real numbers; see for example \cite{APWW98, Bug11, Coons13, TL1, PDS, PF2, keijo15} and references therein. In 2016, Han \cite{Han16} introduced the Hankel continued fraction which is a powerful tool for evaluating Hankel determinants. By using the Hankel continued fractions, Bugeaud, Han Wen and Yao \cite{BHWY16} characterized the irrationality exponents of values of certain degree two Mahler functions at rational points. Recently, Guo, Han and Wu \cite{GHW20} fully characterized apwwenian sequences, that is $\pm 1$ sequences whose Hankel determinants $H_{0,n}$ satisfying $H_{0,n}/2^{n-1}\equiv 1\pmod 2$ for all $n\geq 1$.
	
	However, the Hankel determinants of other low complexity sequences, such as Sturmian sequences, are rarely known. Kamae, Tamura and Wen \cite{TAMURA} explicitly evaluated the Hankel determinants of the Fibonacci word. Tamura \cite{TAM} extended this result to infinite words generated by the substitutions $a\to a^kb, b\to a$ ($k\ge 1$). In this paper, we study the Hankel determinants of the sequence generated by the substitution
	\[\tau: 1\to 101,\, 0\to 1.\] 
	Denote by $\mathbf{s}=(s_{n})_{n\ge 0}=\lim_{n\to\infty}\tau^{n}(1)$ the fixed point of $\tau$. It follows from \cite[Proposition 2.1]{TW03} that $\mathbf{s}$ is a Sturmian sequence.  
	
	\begin{figure}[htbp]
    	\centering
    	\includegraphics[height=\textwidth, trim={1.8cm 1.8cm 1.8cm 1.8cm},clip,angle =90]{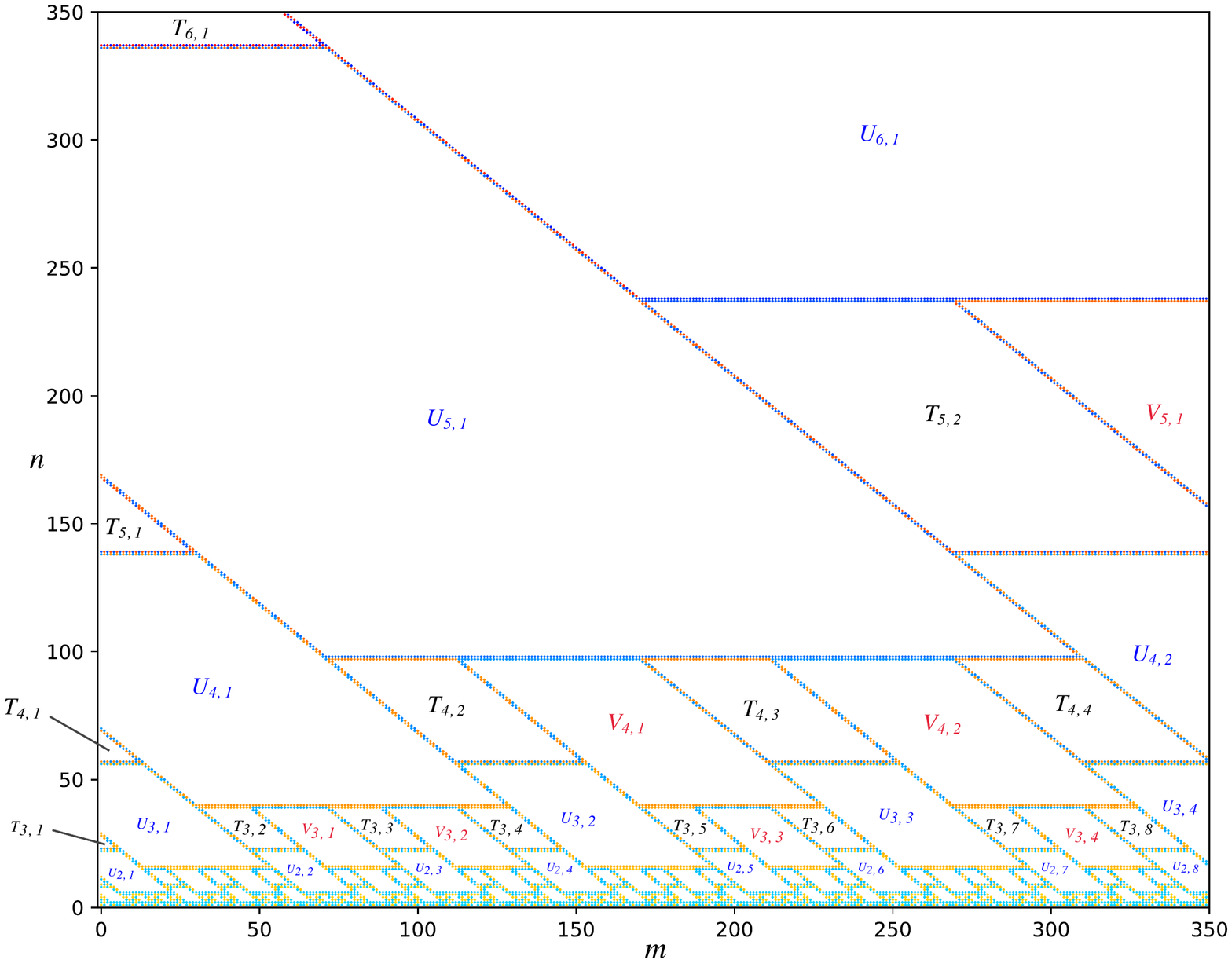}
    	\caption{Visualization of the Hankel determinants $H_{n,m}$ ($0\leq n,m\leq 350$).}
    	\label{fig:intro}
	\end{figure}
	
	We give the explicit values of Hankel determinants $H_{m,n}$ for the sequence $\mathbf{s}$ for all $m\ge 0$ and $n\ge 1$. The distribution of first values of $H_{m,n}$ can be seen from Figure \ref{fig:intro}, where we use different colors to indicate different values of $H_{m,n}$. In particular, the white color indicates that the Hankel determinants are zero. The zeros of the Hankel determinants (together with the non-zero boundaries) form three types of parallelograms labelled by $U_{k,i}$, $V_{k,i}$ and $T_{k,i}$ (for detailed definitions, see Section \ref{sec:3}). In fact, those parallelograms are disjoint and they tile the lattices in the first quadrant; see Proposition \ref{lem20} in Section \ref{sec:3}. This nice property allows to evaluate the Hankel determinants $H_{m,n}$ according to the parallelograms. 
	
	To state our main result, we need some notations. Let $U_k=\cup_{i\ge 1}U_{k,i}$, $V_{k}=\cup_{i\ge 1}V_{k,i}$ and $T_{k}=\cup_{i\ge 1}T_{k,i}$. The integer sequence $(f_n)_{n\ge 0}$ is given by $f_{2j}=|\tau^{j}(1)|$ and $f_{2j+1}=|\tau^{j}(10)|$ for all $j\ge 0$, where $|w|$ denotes the number of digits in the word $w$. To avoid repeating lengthy definitions, please see Section \ref{sec:2} for the truncated $f$-representation $\Phi_k(\cdot)$ and see Section \ref{sec:3} for the sequences $(\alpha_i)_{i\ge 1}$, $(\beta_i)_{i\ge 1}$ and $(\gamma_i)_{i\ge 1}$. The Hankel determinants $H_{m,n}$ for $(m,n)$ in parallelograms of type $U_{k,i}$ (resp.  $V_{k,i}$ and $T_{k,i}$) are given in the following results.

	\begin{theorem} \label{thm1}
		Let $k\ge 0$. For all $(m,n)\in U_{k}$, 
		\begin{enumerate}
			\item when $n=f_{2k+3}-1$, $H_{m,n}=(-1)^{k+1}(-1)^{\frac{f_{2k+5}}{2}-\Phi_{k+1}(m+n)} \cdot\frac{f_{2k+1}}{2}$;
			\item when $n=f_{2k}$, 
			$H_{m,n}=(-1)^{k+1}(-1)^{\frac{f_{2k+2}-1}{2}}\cdot \frac{f_{2k+1}}{2}$;
			\item when $f_{2k}<n<f_{2k+3}-1$, if  $m+n=\alpha_i-f_{2k+2}+1$ or $\alpha_i$ for some $i\ge 1$, then 
			\[H_{m,n}= -(-1)^{(f_{2k+3}-n)k}(-1)^{\frac{(f_{2k+3}-1-n)(f_{2k+3}-2-n)}{2}}\cdot \frac{f_{2k+1}}{2};\]
			otherwise $H_{m,n}=0$.
		\end{enumerate}
	\end{theorem}
	
    \begin{theorem} \label{thm2}
		Let $k\ge 0$. For all $(m,n)\in V_{k}$,
		\begin{enumerate}
			\item when $n=f_{2k+2}-1$, $H_{m,n}=(-1)^{\frac{f_{2k+2}+f_{2k+1}-3}{2}}\cdot \frac{f_{2k+1}}{2}$;
			\item when $n=f_{2k}$,
			$H_{m,n}=(-1)^{k+1}(-1)^{\frac{f_{2k+2}-1}{2}}\cdot \frac{f_{2k+1}}{2}$;
			\item when $f_{2k}<n<f_{2k+2}-1$, if $m+n=\beta_i+1$ or $\beta_i+f_{2k+1}$ for some $i\ge 1$, then \[H_{m,n}=-(-1)^{(f_{2k+2}-n)(k+1)}(-1)^{\frac{(f_{2k+2}-1-n)(f_{2k+2}-2-n)}{2}}\cdot \frac{f_{2k+1}}{2};\]
			otherwise $H_{m,n}=0$.
		\end{enumerate}
	\end{theorem}
	
    \begin{theorem} \label{thm3}
		Let $k\ge 0$. For all $(m,n)\in T_{k}$,
		\begin{enumerate}
			\item when $n=f_{2k+2}-1$, $H_{m,n}=(-1)^{\frac{f_{2k}-1}{2}}\cdot f_{2k}$;
			\item when $n=f_{2k+1}$, 
			$H_{m,n}=(-1)^{\Phi_k(m+n)-\frac{f_{2k+1}}{2}-1}\cdot f_{2k}$;
			\item when $f_{2k}<n<f_{2k+2}-1$, if $m+n=\gamma_i-f_{2k}+1$ or $\gamma_i$ for some $i\ge 1$, then \[H_{m,n}=(-1)^{(f_{2k+2}-1-n)k}(-1)^{\frac{(f_{2k+2}-1-n)(f_{2k+2}-2-n)}{2}}(-1)^{\frac{f_{2k}-1}{2}}\cdot f_{2k};\]
			otherwise $H_{m,n}=0$.
		\end{enumerate}
	\end{theorem}

	The paper is organized as follows. In Section 2, we introduce the $f$-representation of positive integers and give a criterion (Proposition \ref{prop:2}) to determine $s_n$ according the $f$-representation of $n$. This criterion leads us to the key ingredient (Theorem \ref{lem3}) in calculating the Hankel determinants. Then we introduce the truncated $f$-representation which is essential in describing the parallelograms. In Section \ref{sec:3}, we show that the parallelograms $U_{k,i}$, $V_{k,i}$ and $T_{k,i}$ tile all the integer points in the first quadrant. In Section \ref{sec:4}, we first show that the Hankel determinants vanish when $(m,n)$ is inside a parallelogram of those three types. Next we show the relations of Hankel determinants $H_{m,n}$ on the boundary of a parallelogram $U_{k,i}$ (or $V_{k,i}$, $T_{k,i}$) for a given $k$ and $i$. Finally, for any $k\ge 0$, we describe the relation of values of Hankel determinants for parallelograms $U_{k,i}$ for all $i\ge 1$. In Section \ref{sec:5}, we give the expressions for Hankel determinants on the boundary of $U_{k,i}$ (or $V_{k,i}$, $T_{k,i}$) for all $k\ge 0$. In the last section, we formulate and prove our main results.

    \section{Some properties of the sequence \texorpdfstring{$\mathbf{s}$}{\textbf{s}}}\label{sec:2}
	In this section, we first introduce the $f$-representation of positive integers according to the sequence $(f_n)_{n\ge 0}$. By understanding the occurrences of $0$s in the sequence $\mathbf{s}$, we prove a key result (Proposition \ref{prop:2}) which can determine $s_n$ according to the $f$-representation of $n$. Then we give the essential result (Theorem \ref{lem3}). In subsection \ref{sec:2.3}, we introduce the truncated $f$-representation which is useful in determining the parallelograms. In section \ref{sec:2.4}, we investigate to some sub-sequences of $(f_{n})_{n\ge 0}$ which are need in evaluating the coefficients of the Hankel determinants. Then we characterize two sub-sequences of $\mathbf{s}$ which helps us understand $H_{m,n}$.

    \subsection{The occurrence of 0's in \texorpdfstring{$\mathbf{s}$}{\textbf{s}}.}
    We introduce an auxiliary sequence $(f_n)_{n\ge 0}$ to determine the positions of $0$'s. For all $n\geq 0$, we define 
    \[f_{2n}=|\tau^{n}(1)|\quad \text{and}\quad f_{2n+1}=|\tau^{n}(10)|.\] 
    Then $f_0=1$, $f_1=2$, and for all $n\ge 0$,  
    \begin{equation}\label{eq:f-rec}
        \begin{cases}
            f_{2n+2} = f_{2n} + f_{2n+1},\\
            f_{2n+3} = f_{2n} + f_{2n+2}.\\
        \end{cases}
    \end{equation}
    The first values are \[(f_n)_{n\ge 0}=(1,\, 2,\, 3,\,4,\, 7,\, 10,\, 17,\, 24,\, 41, \, 58,\,99,\,140,\,239,\,338,\,577,\,816,\,\dots).\]
    Since $(f_n)_{n\ge 0}$ is an increasing non-negative integer sequence, it is a numeration system in the following sense.

    \begin{lemma} [Theorem 3.1.1 \cite{AS03}] \label{lem2}
        Let $u_0 < u_1 < u_2 < \dots$ be an increasing sequence of integers with
        $u_0=1$. Every non-negative integer $N$ has exactly one representation of the form $\sum_{0\leq i\leq r}a_iu_i$ where $a_r\neq 0$, and for $i\ge 0$, the digits $a_i$ are non-negative integers satisfying the inequality
        \[a_0u_0 +a_1u_1 +\dots +a_iu_i < u_{i+1}.\]
    \end{lemma}

    \begin{proposition}\label{prop:1}
        Every integer $n\ge 0$ can be uniquely expressed as $n=\sum_{0\leq i\leq r}a_if_i$ with $a_i\in\{0,1\}$, $a_r\neq 0$, and         
        \begin{equation}\label{eq:1}
            \begin{cases}
                a_ia_{i+1}=0 \text{ for all } 0\leq i< r,\\
                a_{i}a_{i+2}=0 \text{ for all even numbers }  0\leq i < r-1.
            \end{cases}
        \end{equation}  
    \end{proposition}
    \begin{proof}
        Suppose $a_i\in\{0,1\}$. By Lemma \ref{lem2}, we only need to show that $a_0f_0+a_1f_1+\dots+a_tf_t < f_{t+1}$ for all $t$ if and only if the condition \eqref{eq:1} holds.

        {\bf The `only if' part.} Suppose there is an index $i$ such that $a_ia_{i+1}=1$. Then 
        \[a_0f_0+a_1f_1+\dots+a_if_i+a_{i+1}f_{i+1} \geq f_i+f_{i+1} \geq f_{i+2},\]
        which is a contradiction for $t=i+1$. Suppose there is an even index $i$ such that $a_ia_{i+2}=1$. Then 
        \[a_0f_0+a_1f_1+\dots+a_if_i+a_{i+1}f_{i+1}+a_{i+2}f_{i+2} \geq f_i+f_{i+2} = f_{i+3},\]
        which is a contradiction for $t=i+2$.

        {\bf The `if' part.} Suppose the condition \eqref{eq:1} holds. When $t$ is odd, the maximum possible value of $a_0f_0+a_1f_1+\dots+a_tf_t$ occurs when $a_ta_{t-1}\dots a_0=1010\dots10$, and this maximum value $f_1+f_3+\dots f_{t-2}+f_{t}=f_{t+1}-1$. When $t$ is even, the maximum possible value of $a_0f_0+a_1f_1+\dots+a_tf_t$ occurs when $a_ta_{t-1}\dots a_0=1001010\dots10$. In this case, the maximum value is $f_1+f_3+\dots +f_{t-5} + f_{t-3} + f_{t} = f_{t+1}-1$.
    \end{proof}

    \begin{definition}[$f$-representation]
        Let  $n\ge 0$ be an integer. We call the representation $n=\sum_{0\leq i\leq r}a_if_i$ in Proposition \ref{prop:1} the \emph{$f$-representation} of $n$. We also write $n=\sum_{i=0}^{+\infty}a_if_i$ where $a_i=0$ for all $i>r$. In the case that we need to emphasize that $a_i$ depends on $n$, we write $a_i=a_i(n)$ as a function of $n$. 
    \end{definition}

    \begin{proposition}\label{prop:2}
        For any integer $n\geq 0$ with the $f$-representation $\sum_{i=0}^{r}a_i(n)f_i$, we have $s_n=0$ if and only if $a_0(n)=1$. 
    \end{proposition}
    \begin{proof}
        One can verify directly that the result holds for all $n<f_4=7$. Assume that the result holds for $n< f_{2k}$ where $k\geq 2$. We only need to prove it for all $f_{2k}\leq n< f_{2k+2}$.
        
        Suppose $f_{2k}\le n< f_{2k+1}$. One has $a_{2k}(n)=1$ and hence $a_0(n-f_{2k})=a_0(n)$. Note that $f_{2k}=|\tau^{k}(1)|$ and 
        \begin{equation}\label{eq:prop:2}
            s_0s_1\dots s_{f_{2k+2}-1}=\tau^{k+1}(1)=\tau^{k}(1)\tau^{k}(0)\tau^{k}(1).
        \end{equation} We see that $s_n$ is the $(n+1)$-th letter of $\tau^{k+1}(1)$ and it is also the $(n+1-f_{2k})$-th letter of $\tau^{k}(0)=\tau^{k-1}(1)$. Consequently, $s_{n}=s_{n-f_{2k}}$. 
        Since \[n-f_{2k}< f_{2k+1}-f_{2k} =f_{2k-2},\] by the inductive assumption, we have $s_{n-f_{2k}}=0$ if and only if $a_0(n-f_{2k})=1$.  Therefore, $s_n=0$ if and only if $a_0(n)=1$.

        Suppose $f_{2k+1}\le n< f_{2k+2}$. In this case $a_{2k+1}(n)=1$ and $a_0(n-f_{2k+1})=a_0(n)$. Since $|\tau^{k}(10)|=f_{2k+1}$, it follows from \eqref{eq:prop:2} that $s_n=s_{n-f_{2k+1}}$. Note that 
        \[ n-f_{2k+1}< f_{2k+2}-f_{2k+1} =f_{2k}.\]
        By the inductive assumption, $s_{n-f_{2k+1}}=0$ if and only if $a_{0}(n-f_{2k+1})=1$ which implies the result also holds for all $f_{2k+1}\le n< f_{2k+2}$.
    \end{proof}

\subsection{Comparing digits in the sequence \texorpdfstring{$\mathbf{s}$}{s} with a fixed gap}
We introduce the truncated $f$-representations (of positive integers) which are useful in telling two digits with a fixed gap in $\mathbf{s}$ are equal or not.
\begin{definition}(Truncated $f$-representation)
	Let $n\geq 0$ be an integer with the $f$-representation $\sum_{i=0}^{+\infty}a_i(n)f_i$. For all integers $k\ge 0$, the truncated $f$-representation of $n$ is 
	\[\Phi_{k}(n):=\sum_{i=0}^{2k+2}a_i(n)f_i.\]
\end{definition}

The next lemma gives a criterion that when two digits (with a fixed gap) in $\mathbf{s}$ are equal by using their positions.
    \begin{theorem}\label{lem3}
        Let $n\geq 0$ be an integer with the $f$-representation $\sum_{i=0}^{+\infty}a_i(n)f_i$. Then 
        \begin{enumerate}
            \item[(i)] for all $k\ge 0$, $s_{n+f_{2k}}\neq s_n$ if and only if $\Phi_{k}(n)\in\{\frac{f_{2k+1}}{2}, \frac{f_{2k+1}}{2}-1\}$;
            \item[(ii)] for all $k\ge 1$, $s_{n+f_{2k+1}}\neq s_n$ if and only if $\Phi_{k}(n)\in\{\frac{f_{2k+3}}{2}, \frac{f_{2k+3}}{2}-1, \frac{f_{2k+3}}{2}+f_{2k}, \frac{f_{2k+3}}{2}+f_{2k}-1\}$. 
        \end{enumerate}
	\end{theorem}
	
    \begin{proof}
    (i) We prove by induction on $k$. When $k=0$, by Proposition \ref{prop:1}, there are only four possible values for $a_0(n)a_1(n)a_2(n)$. By Eq.~\eqref{eq:f-rec}, we have
    \begin{center}
        \begin{tabular}{c|l|l|l|l}
            \hline
            $\Phi_{0}(n)$ & 0 & 1 & 2 & 3\\
            \hline
            $a_0(n)a_1(n)a_2(n)$  & 000 & 100 & 010 & 001 \\
            \hline
            $a_0(n+f_0)$ & 1 & 0 & 0 & 0\\
            \hline
        \end{tabular}.
    \end{center}
    Then we see that $a_0(n)\neq a_0(n+f_0)$ if and only if $\Phi_{0}(n)=0 $ or $1$. The result holds for $k=0$.

    When $k=1$, note that $\Phi_{1}(n)\leq \sum_{i=0}^{4}f_i<f_5=10$. We see 
    \begin{center}\small
        \begin{tabular}{c*{10}{|l}}
            \hline
            $\Phi_{1}(n)$ & 0 & 1 & 2 & 3 & 4 & 5 & 6 & 7 & 8 & 9\\ \hline
            $a_0(n)\dots a_4(n)$ & 00000 & 10000 & 01000 & 00100 & 00010 & 10010 & 01010 & 00001 & 10001 & 01001\\ \hline
            $a_0(n+f_2)$ & 0 & 0 & 1 & 0 & 0 & 1 & 0 & 0 & 1 & 0\\ \hline
        \end{tabular}.
    \end{center}
    Then we have $a_0(n)\neq a_0(n+f_2)$ if and only if $\Phi_{1}(n)=1$ or $2$, that is $\frac{f_3}{2}-1$ or $\frac{f_3}{2}$. The result also holds for $k=1$.
    
    Now assume that the result holds for all $0\leq k<\ell$ with $\ell\ge 2$. We prove it for $k=\ell$. Let $w=a_{2\ell-2}(n)a_{2\ell-1}(n)\dots a_{2\ell +2}(n)$ and $v=a_{2\ell-2}(n+f_{2\ell})a_{2\ell-1}(n+f_{2\ell})\dots a_{2\ell +1}(n+f_{2\ell})$. According to Proposition \ref{prop:1}, $w$ can take only 10 different values.  
    \begin{table}[htbp]
        \caption{Values of $v$.}\label{table:1}
        \begin{tabular}{c*{10}{|l}}
            \hline
            $w$ & 00000 & 10000 & 01000 & 00100 & 00010 & 10010 & 01010 & 00001 & 10001 & 01001\\ \hline
            $v$ & 0010 & 0001 & ? & 0101 & 0000 & 1000 & 0100 & 0000 & 1000 & 0100 \\ \hline
        \end{tabular}
    \end{table}
    While $w\neq 01000$, one can determine $v$ directly using Eq.~\eqref{eq:f-rec}; thus in these cases, $a_{i}(n+f_{2\ell})=a_i(n)$ for all $0\leq i\leq 2\ell -3$; see Table \ref{table:1}.  
    For instance, when $w=10010$, 
    \begin{align*}
        n + f_{2\ell} & = \left(\sum_{i=0}^{2\ell-3}a_i(n)f_i + f_{2\ell-2}+f_{2\ell+1}+\sum_{i=2\ell+3}^{+\infty}a_i(n)f_i\right) +f_{2\ell}\\
        & = \left(\sum_{i=0}^{2\ell-3}a_i(n)f_i + f_{2\ell-2}\right)+\left(f_{2\ell+2}+\sum_{i=2\ell+3}^{+\infty}a_i(n)f_i\right).
    \end{align*}
    Hence one can see that $a_0(n+f_{2\ell})=a_0(n)$.

    When $w=01000$, set $n'=\sum_{i=0}^{2\ell-2}a_i(n)f_i$. Then $a_0(n)=a_0(n')$ and  
    \begin{align*}
            n + f_{2\ell} & = \left(n'+ f_{2\ell-1} + \sum_{i=2\ell +3}^{+\infty}a_i(n)f_i \right)+ f_{2\ell}\\
            & = (n' + f_{2\ell -4}) + f_{2\ell +1} +\sum_{i=2\ell +3}^{+\infty}a_i(n)f_i.  \quad \text{(by Eq.~\ref{eq:f-rec})}
    \end{align*}
    Noticing that $n' + f_{2\ell -4}<f_{2\ell -2}$, we have $a_i(n+f_{2\ell})=a_i(n'+f_{2\ell-4})$ for all $0\leq i\leq 2\ell -2$. In particular, $a_0(n+f_{2\ell})=a_0(n'+f_{2\ell-4})$. By Proposition \ref{prop:2} and the inductive assumption, 
    \begin{align*}
        a_0(n+f_{2\ell})\neq a_0(n) & \iff a_0(n'+f_{2\ell-4})\neq a_0(n')\\
         & \iff n'\in\left\{\frac{f_{2\ell-3}}{2},\frac{f_{2\ell-3}}{2}-1\right\}\\
        & \iff \Phi_{\ell}(n) = n'+ f_{2\ell -1} \in\left\{\frac{f_{2\ell-3}}{2}+ f_{2\ell -1},\frac{f_{2\ell-3}}{2}+ f_{2\ell -1}-1\right\}.
    \end{align*}
    By Eq.~\eqref{eq:f-rec}, we have 
    \begin{align}
        \frac{f_{2\ell+1}}{2} & = (f_{2\ell-2} + f_{2\ell})/2 \nonumber\\
        & = (f_{2\ell-2} + f_{2\ell-2} + f_{2\ell-1})/2\nonumber\\
        & = f_{2\ell-2} + \frac{f_{2\ell-1}}{2}\nonumber\\
        & = f_{2\ell-2} + f_{2\ell-4}  + \frac{f_{2\ell-3}}{2}\nonumber\\
        & = f_{2\ell-1} + \frac{f_{2\ell-3}}{2}. \label{eq:f-2-odd}
    \end{align}
    Then we obtain that 
    \[a_0(n+f_{2\ell})\neq a_0(n) \iff \Phi_{\ell}(n)\in\left\{\frac{f_{2\ell+1}}{2} ,\frac{f_{2\ell+1}}{2} -1\right\}.\]
    It follows from Proposition \ref{prop:2} that the result holds for $k=\ell$. 

    \medskip
    (ii) For any $k\ge 1$, let $u=a_{2k}(n)a_{2k+1}(n)a_{2k+2}(n)$. It follows from  Proposition \ref{prop:2} that $u\in\{100,010,001\}$. The proof is divided into the following three cases.
    \begin{itemize}
        \item When $u=001$, we also have $a_{2k+3}(n)=a_{2k+4}(n)=0$. So
        \begin{align*}  
            n + f_{2k+1} & = \left(\sum_{i=0}^{2k-1}a_i(n)f_i + f_{2k+2} + \sum_{i=2k+5}^{+\infty}a_i(n)f_i\right) + f_{2k+1}\\
            & = \left(\sum_{i=0}^{2k-1}a_i(n)f_i + f_{2k-2}\right) + \left(f_{2k+3}+ \sum_{i=2k+5}^{+\infty}a_i(n)f_i \right).
        \end{align*}    
        Let $n'=\sum_{i=0}^{2k-1}a_i(n)f_i$. Then $a_0(n)=a_0(n')$. Since $n'<f_{2k}$ and $n'+f_{2k-2}<f_{2k+1}$, we have $a_i(n'+f_{2k-2})=a_i(n+f_{2k+1})$ for all $0\leq i\leq 2k$. Thus 
        \begin{align*}
            a_0(n+f_{2k+1})\neq a_0(n) & \iff  a_i(n'+f_{2k-2})\neq a_0(n')\\
            & \iff n'=\sum_{i=0}^{2k}a_i(n')f_i\in \left\{\frac{f_{2k-1}}{2},\frac{f_{2k-1}}{2}-1\right\}
        \end{align*}
        where in the last step we use Theorem \ref{lem3}(i). By Eq.~\eqref{eq:f-rec} and Eq.~\eqref{eq:f-2-odd},
        \[\frac{f_{2k-1}}{2}+f_{2k+2} = \frac{f_{2k-1}}{2}+f_{2k+1}+f_{2k} = \frac{f_{2k+3}}{2}+f_{2k}.\] 
        So when $u=001$, $a_0(n+f_{2k+1})\neq a_0(n)$ if and only if 
        \[\Phi_{k}(n) = n'+f_{2k+2} \in \left\{\frac{f_{2k+3}}{2}+f_{2k}, \frac{f_{2k+3}}{2}+f_{2k}-1\right\}.\]

        \item Suppose $u=010$. Applying Eq.~\eqref{eq:f-rec} twice, we obtain that \[2f_{2k+1}=f_{2k-2}+f_{2k}+f_{2k+1}=f_{2k-2}+f_{2k+2}.\] Then 
        \begin{align*}  
            n + f_{2k+1} & = \left(\sum_{i=0}^{2k-1}a_i(n)f_i + f_{2k+1} + \sum_{i=2k+3}^{+\infty}a_i(n)f_i\right) + f_{2k+1}\\
            & = \left(\sum_{i=0}^{2k-1}a_i(n)f_i + f_{2k-2}\right) + \left(f_{2k+2}+ \sum_{i=2k+3}^{+\infty}a_i(n)f_i \right).
        \end{align*} 
        Let $n'=\sum_{i=0}^{2k-1}a_i(n)f_i$. Using Theorem \ref{lem3}(i), the same argument as in the case $u=001$ leads us to the fact that 
        \begin{align*}
            a_0(n+f_{2k+1})\neq a_0(n) & \iff n'=\sum_{i=0}^{2k}a_i(n')f_i\in \left\{\frac{f_{2k-1}}{2},\frac{f_{2k-1}}{2}-1\right\}\\
            & \iff \Phi_{k}(n) = n'+f_{2k+1} \in \left\{\frac{f_{2k+3}}{2}, \frac{f_{2k+3}}{2}-1\right\}.
        \end{align*}
        \item When $u=100$, we have $a_{2k-2}(n)=a_{2k-1}(n)=0$. Then 
        \begin{align*}  
            n + f_{2k+1} & = \left(\sum_{i=0}^{2k-3}a_i(n)f_i + f_{2k} + \sum_{i=2k+3}^{+\infty}a_i(n)f_i\right) + f_{2k+1}\\
            & = \left(\sum_{i=0}^{2k-3}a_i(n)f_i \right) + \left(f_{2k+2}+ \sum_{i=2k+3}^{+\infty}a_i(n)f_i \right)
        \end{align*} 
        which implies that $a_0(n+f_{2k+1})=a_0(n)$. \qedhere
    \end{itemize}
	\end{proof}
	
\subsection{Integers with the same truncated \texorpdfstring{$f$}{f}-representation}\label{sec:2.3}
To apply Theorem \ref{lem3}, we need to investigate the integers of the same truncated $f$-representations. The following two lemmas (Lemma \ref{lem6} and Lemma \ref{lem7}) serve for this purpose.

For all $k\ge 0$, denote
\[E'_k=\left\{x\in\mathbb{N}\,:\,\Phi_{k}(x)=\frac{f_{2k+3}}{2}\right\}\quad
\text{and}\quad
E''_k= \left\{x\in\mathbb{N}\,:\,\Phi_{k}(x)=\frac{f_{2k+3}}{2}+f_{2k}\right\}.\]
Let $E_k=E'_k\cup E''_k=(x^{(k)}_j)_{j\ge 1}$ 
where $x^{(k)}_1<x^{(k)}_2<x^{(k)}_3<\dots$. The first values of $E_k$ are 
\begin{align*}
	(x^{(k)}_j)_{j\ge 1} = & \left(\frac{f_{2k+3}}{2},\,\frac{f_{2k+3}}{2}+f_{2k},\,\frac{f_{2k+3}}{2}+f_{2k+3},\,\frac{f_{2k+3}}{2}+f_{2k+4},\,\frac{f_{2k+3}}{2}+f_{2k+5},\,\right.\\
	& \qquad\left.\frac{f_{2k+3}}{2}+f_{2k}+f_{2k+5},\,\frac{f_{2k+3}}{2}+f_{2k+3}+f_{2k+5},\,\frac{f_{2k+3}}{2}+f_{2k+6},\,\dots\right).
\end{align*}

\begin{lemma}  \label{lem6}
	Let $k\ge 0$ and $x\in E'_k$ with $x=x^{(k)}_j$ for some $j\ge 2$. Then $x-f_{2k+2}=x^{(k)}_{j-1}\in E_k$.
\end{lemma}
\begin{proof}
	Let $x\in E''_{k}$ with $x=x_{j}^{(k)}$ for some $j\ge 2$. Note that $\frac{f_{2k+3}}{2}+f_{2k}=f_{2k+2}+\frac{f_{2k-1}}{2}$. By Proposition \ref{prop:1},  have $a_{2k+3}(x)=a_{2k+4}(x)=0$. When $0<b<f_{2k}-\frac{f_{2k-1}}{2}$, we see \[\Phi_{k}(x+b)=\left(\frac{f_{2k+3}}{2}+f_{2k}\right)+b<f_{2k+3}\] which implies that $x+b\notin E_k$. When $f_{2k}-\frac{f_{2k-1}}{2}\leq b <f_{2k+2}$, \[\Phi_{k}(x+b)=\left(\frac{f_{2k+3}}{2}+f_{2k}\right)+b - f_{2k+3} <\frac{f_{2k+3}}{2}.\] So $x+b\notin E_k$. Since $\Phi_{k}(x+f_{2k+2})=\frac{f_{2k+3}}{2}$, we have $x+f_{2k+2}=x_{j+1}^{(k)}\in E'_k$.

	Let $x\in E'_{k}$ with $x=x_{j}^{(k)}$ for some $j\ge 1$. According to Proposition \ref{prop:1}, $a_{2k+3}(x)a_{2k+4}(x)=00$, $10$ or $01$, which can be divided into two sub-cases.
	\begin{itemize}
		\item $a_{2k+3}(x)a_{2k+4}(x)=00$. For $0<b\leq f_{2k}$, we see $\Phi_{k}(x+b)=\Phi_{k}(x)+b=\frac{f_{2k+3}}{2}+b$. Thus $x+f_{2k}=x_{j+1}^{(k)}\in E''_k$. 
		\item $a_{2k+3}(x)a_{2k+4}(x)=01$ or $10$. For $0<b< f_{2k}$, we see \[\Phi_{k}(x+b)=\Phi_{k}(x)+b=\frac{f_{2k+3}}{2}+b<\frac{f_{2k+3}}{2}+f_{2k}.\] Thus $x+b\notin E_k$. For $f_{2k}\leq b<f_{2k+2}$, we have \[\Phi_{k}(x+b)=\frac{f_{2k+3}}{2}+b-f_{2k+2}\in \left(\frac{f_{2k-1}}{2},\,\frac{f_{2k+3}}{2}\right)\] which yields that $x+b\notin E_k$. Noting that $\Phi_k(x+f_{2k+2})=\Phi_k(x)=\frac{f_{2k+3}}{2}$, we obtain that $x+f_{2k+2}=x_{j+1}^{(k)}\in E'_k$.
	\end{itemize}  

	From the above argument, we see that if $x=x_{j}^{(k)}\in E'_k$ for some $j\ge 2$, then either $x-f_{2k+2}=x_{j-1}^{(k)}\in E''_k$ or $x-f_{2k+2}=x_{j-1}^{(k)}\in E'_k$ with $a_{2k+3}(x-f_{2k+2})a_{2k+4}(x-f_{2k+2})\neq 00$. The result holds.
\end{proof}

\begin{remark} \label{re-gap-E}
	From the proof of Lemma \ref{lem6}, we see the gaps between two adjacent elements in $E_k$ are $f_{2k}$ and $f_{2k+2}$. That is $x_{j+1}^{(k)}-x_{j}^{(k)}=f_{2k}$ or $f_{2k+2}$ for all $j\ge 1$. Moreover, the gaps between two adjacent elements in $E'_k$ are $f_{2k+2}$ and $f_{2k+3}$. 
\end{remark}

For all $k\ge 0$, let
\[
F_k=\left\{y\in\mathbb{N}\,:\,\Phi_{k}(y)=\frac{f_{2k+1}}{2}\right\}=(y^{(k)}_j)_{j\ge 1}\quad\text{and}\quad
F'_k=\left\{y\in\mathbb{N}\,:\,\Phi_{k+1}(y)=\frac{f_{2k+1}}{2}\right\}
\]
where $y^{(k)}_1<y^{(k)}_2<y^{(k)}_3<\dots$. Write $F''_k=F_k-F'_k$. The first values of $F_k$ are 
\begin{align*}
	(y_{j}^{(k)})_{j\ge 1} = & \left(\frac{f_{2k+1}}{2},\, \frac{f_{2k+1}}{2}+f_{2k+3},\, \frac{f_{2k+1}}{2}+f_{2k+4},\, \frac{f_{2k+1}}{2}+f_{2k+5},\,\right. \\
&\qquad \left.\frac{f_{2k+1}}{2}+f_{2k+5}+f_{2k+3},\,\frac{f_{2k+1}}{2}+f_{2k+6},\,\dots \right).
\end{align*}

\begin{lemma}  \label{lem7}
	For any $y\in F_k$ with $y=y^{(k)}_j$ for some $j\ge 1$, we have 
	\[y^{(k)}_{j+1}=\begin{cases}
		y+f_{2k+3}, & \text{if }y\in F'_{k};\\
		y+f_{2k+2}, & \text{if }y\in F''_{k}.
	\end{cases}\]
\end{lemma}
\begin{proof}
	We prove the result by giving the construction of $F_k$. It clear that $y_1^{(k)}=\frac{f_{2k+1}}{2}$. Now suppose  $y=y_{j}^{(k)}\in F_k$ where $j\ge 1$.  
	According to Proposition \ref{prop:1}, we see $a_{2k+3}(y)a_{2k+4}(y)=00$, $ 01$ or $01$.  
	\begin{itemize}
		\item $a_{2k+3}(y)a_{2k+4}(y)=00$, i.e., $y\in F'_{k}$. Note that $\frac{f_{2k+1}}{2}=f_{2k-1}+\frac{f_{2k-3}}{2}$. For $0<b< f_{2k+3}-\frac{f_{2k+1}}{2}$, we have $\Phi_k(y+b)=\frac{f_{2k+1}}{2}+b$, so $y+b\notin F_k$. For $f_{2k+3}-\frac{f_{2k+1}}{2}\leq b < f_{2k+3}$,  \[\Phi_k(y+b)=\frac{f_{2k+1}}{2}+b-f_{2k+3}<\frac{f_{2k+1}}{2},\]
		so $y+b\notin F_{k}$. Since $\Phi_{k}(y+f_{2k+3})=\Phi_{k}(y)$, we obtain that $y+f_{2k+3}=y_{j+1}^{(k)}\in F_{k}-F_{k}'$.
		\item $a_{2k+3}(y)a_{2k+4}(y)= 10$ or $01$. For $0<b<f_{2k+2}-\frac{f_{2k+1}}{2}$, we have $\Phi_k(y+b)=\frac{f_{2k+1}}{2}+b$, so $y+b\notin F_k$. For $f_{2k+2}-\frac{f_{2k+1}}{2}\leq b <f_{2k+2}$, since \[\Phi_k(y+b)=\frac{f_{2k+1}}{2}+b-f_{2k+2}<\frac{f_{2k+1}}{2},\] we also have $y+b\notin F_{k}$. It follows from $\Phi_k(y+f_{2k+2})=\Phi_{k}(y)$ that $y+f_{2k+2}=y_{j+1}^{(k)}\in F_k$.
	\end{itemize}
	The result follows from the above two sub-cases.
\end{proof}
\begin{remark}
	From the proof of Lemma \ref{lem7}, we see the gaps between two adjacent elements in $F_k$ are $f_{2k+2}$ and $f_{2k+3}$. That is $y_{j+1}^{(k)}-y_{j}^{(k)}=f_{2k+2}$ or $f_{2k+3}$ for all $j\ge 1$. Moreover, the gaps between two adjacent elements in $F''_k$ are $f_{2k+2}$ and $f_{2k+4}$. 
\end{remark}

\subsection{Two subsequences of \texorpdfstring{$\mathbf{s}$}{s}}\label{sec:2.4}
The subsequences $(s_{\frac{f_{2k+1}}{2}})_{k\ge 0}$ and $(s_{\frac{f_{2k+1}}{2}-1})_{k\ge 0}$ can be determined according to the parity of $k$; see Lemma \ref{lem5}. We start with an auxiliary lemma which concerns the parity of $\frac{f_{2k+1}}{2}$.
    \begin{lemma} \label{lem4}
        For all $k\ge 0$,
        \begin{enumerate}
            \item[(i)]  $f_{2k}\equiv\begin{cases}
            1, & \text{if }k\equiv 0 \text{ or }3 \pmod 4,\\
            3, & \text{if }k\equiv 1 \text{ or }2 \pmod 4,
        \end{cases}\pmod 4$,
            \item[(ii)] $f_{2k+1}\equiv
        \begin{cases}
            2, &\text{if }k\text{ is even},\\
            0, &\text{if }k\text{ is odd},\\
        \end{cases}\pmod 4$.
        \end{enumerate}
    \end{lemma}
    \begin{proof}
        (i) Note that $f_0=1$ and $f_{2}=3$. Since $f_{2n}$ is odd for all $n\ge 0$, using Eq.~\eqref{eq:f-rec} twice, we have for all $k\ge 2$, 
        \begin{align*}
            f_{2k} & = f_{2k-2} + f_{2k-1} = 2f_{2k-2} + f_{2k-4} \equiv 2 + f_{2(k-2)}\pmod 4.
        \end{align*}
        The result follows by induction on $k$.

        (ii) The initial value is $f_1=2$. Using Eq.~\eqref{eq:f-rec} and the previous result (i), we have for all $k\ge 1$,
        \[f_{2k+1}=f_{2k}+f_{2k-2}\equiv
        \begin{cases}
            2, & k\equiv 0, 2 \pmod 4,\\
            0, & k\equiv 1, 3 \pmod 4,
        \end{cases}\pmod 4\]
        which is the desired result.
    \end{proof}

	In the calculation of $H_{m,n}$, we need to know $s_n$ explicitly for some $n$. The next lemma determines the values of two sub-sequences $\mathbf{s}$.
    \begin{lemma}\label{lem5}
        For all $k\ge 0$, 
        \[s_{\frac{f_{2k+1}}{2}} = 
        \begin{cases}
            1, & \text{ if } k \text{ is odd},\\
            0, & \text{ if } k \text{ is even},
        \end{cases}\quad \text{and}\quad
        s_{\frac{f_{2k+1}}{2}-1} = 
        \begin{cases}
            0, & \text{ if } k \text{ is odd},\\
            1, & \text{ if } k \text{ is even}.
        \end{cases}\]
    \end{lemma}
    \begin{proof}
        By Eq.~\eqref{eq:f-rec}, we obtain that for all $k\ge 0$,
    \begin{align}
        \frac{f_{2k+1}}{2} & = (f_{2k-2} + f_{2k})/2 \nonumber\\
        & = (f_{2k-2} + f_{2k-2} + f_{2k-1})/2\nonumber\\
        & = f_{2k-2} + \frac{f_{2k-1}}{2} = \cdots = \sum_{i=0}^{k-1}f_{2i} + \frac{f_1}{2}.\label{eq:lem:odd}
    \end{align}
    When $k$ is odd, 
    \begin{align}
        \frac{f_{2k+1}}{2} & = (f_{2k-2}+f_{2k-4}) + (f_{2k-6}+f_{2k-8})+\dots+(f_{4}+f_{2}) + f_0 + \frac{f_1}{2}\notag\\
        & = f_{2k-1} + f_{2k-5} +\dots +f_5 + f_1\quad \text{(by Eq.~\eqref{eq:f-rec})}\notag\\
        & = \sum_{i=0}^{\frac{k-1}{2}}f_{4i+1}.\label{eq:f-odd}
    \end{align}
    When $k\ge 2$ is even,
    \begin{align}
        \frac{f_{2k+1}}{2} & = (f_{2k-2}+f_{2k-4}) + (f_{2k-6}+f_{2k-8})+\dots+(f_{2}+f_{0}) + \frac{f_1}{2}\nonumber\\
        & = f_{2k-1} + f_{2k-5} +\dots +f_3 + \frac{f_1}{2}\quad \text{(by Eq.~\eqref{eq:f-rec})}\nonumber\\
        & = \sum_{i=0}^{\frac{k-2}{2}}f_{4i+3}+ \frac{f_1}{2}
          = \sum_{i=0}^{\frac{k-2}{2}}f_{4i+3} + f_0.\label{eq:f-even}
    \end{align}
    It follows from \eqref{eq:f-odd} and \eqref{eq:f-even} that for all $k\geq 0$,
    \begin{equation*}
        a_0\left(\frac{f_{2k+1}}{2}\right) = 
        \begin{cases}
            0, & \text{ if } k \text{ is odd},\\
            1, & \text{ if } k \text{ is even},
        \end{cases}
    \end{equation*}
    and 
    \begin{equation*}
        a_0\left(\frac{f_{2k+1}}{2}-1\right) = 
        \begin{cases}
            1, & \text{ if } k \text{ is odd},\\
            0, & \text{ if } k \text{ is even}.
        \end{cases}
    \end{equation*}
    Then by Proposition \ref{prop:2}, the result follows.
    \end{proof}

\section{Partition of the lattice}\label{sec:3}
According to the values of the Hankel determinants of $\mathbf{s}$, we tile the integer lattice using the following parallelograms. Given a $k\ge 0$, write the elements in $E'_{k+1}$, $F''_{k}$ and $E'_k$ in ascending order as follows: 
\[E'_{k+1}=(\alpha_i)_{i\ge 1},\quad F''_{k}=(\beta'_{i})_{i\ge 1},\quad E'_{k}=(\gamma_{i})_{i\ge 1}.\]
Moreover, let $\beta_i=\beta'_{i}+f_{2k}$ for all $i\ge 1$. 
We define three different types of parallelograms: for $i\ge 1$, 
\begin{align*}
	U_{k,i} & = \left\{(m,n)\in\mathbb{N}^2\,:\,\, f_{2k}\leq n<f_{2k+3},\, \alpha_{i}-f_{2k+2}<n+m\leq \alpha_i\right\},\\
	V_{k,i} & = \left\{(m,n)\in\mathbb{N}^2\,:\,\, f_{2k}\leq n<f_{2k+2},\, \beta_{i}<n+m\leq \beta_{i}+f_{2k+1}\right\},\\
	T_{k,i} & = \left\{(m,n)\in\mathbb{N}^2\,:\,\, f_{2k+1}\leq n<f_{2k+2},\, \gamma_{i}-f_{2k}<n+m\leq \gamma_{i}\right\};
\end{align*}
see Figure \ref{fig:intro}. 
Let $U_k=\cup_{i\ge 1}U_{k,i}$, $V_{k}=\cup_{i\ge 1}V_{k,i}$ and $T_{k}=\cup_{i\ge 1}T_{k,i}$.

\begin{proposition} \label{lem20}
	The parallelograms $\{U_{k,i}\}$, $\{V_{k,i}\}$, and $\{T_{k,i}\}$ introduce a partition of pairs of positive integers. Namely, $\mathbb{N}\times\mathbb{N}_{\ge 1} = \bigsqcup_{k\geq 0} (U_{k}\sqcup V_{k}\sqcup T_{k})$ where $\sqcup$ denotes the disjoint union.
\end{proposition}
\begin{proof}
	Let $m\ge 0$ and $n\ge 1$ be two integers. Since $(f_k)_{k\ge 0}$ and $(\gamma_{k})_{k\ge 1}$ are two increasing unbounded non-negative integer sequences, there exist $k\ge 0$ and $\ell\ge 1$ such that $f_{2k}\le n <f_{2k+2}$ and $\gamma_{\ell-1}< n+m\leq \gamma_{\ell}$ where $\gamma_0:=0$. The result clearly holds when $\ell=1$. Now we assume that $\ell\ge 2$. From the proof of Lemma \ref{lem6} we see that $\gamma_{\ell}-\gamma_{\ell-1}=f_{2k+2}$ or $f_{2k+3}$ for all $\ell\ge 2$. When $\gamma_{\ell}-f_{2k}<n+m\leq \gamma_{\ell}$, we have 
	\[(m,n)\in\begin{cases}
		U_{k-1}, & \text{ if }f_{2k}\leq n<f_{2k+1};\\
		T_{k}, & \text{ if }f_{2k+1}\leq n<f_{2k+2};
	\end{cases}\] 
	see also Figure \ref{fig:partition}. 
	\begin{figure}[htbp]
		\centering
		\includegraphics[width=\textwidth, trim={0 7cm 0 7cm},clip]{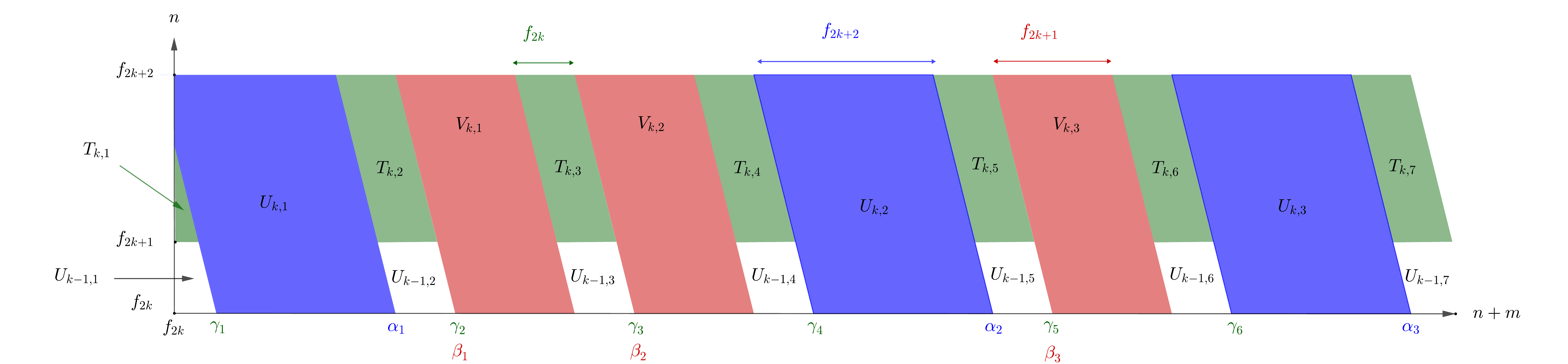}
		\caption{Partition of the strip $[0,+\infty)\times [f_{2k},f_{2k+2})$}
		\label{fig:partition}
	\end{figure}
	When $\gamma_{\ell-1}<n+m\leq \gamma_{\ell}-f_{2k}$, we have the following two cases:

	Case 1: $\gamma_{\ell}-\gamma_{\ell-1}=f_{2k+2}$. In this case, we shall verify that $(m,n)\in V_k$. To do this, we only need to show that $\gamma_{\ell-1}-f_{2k}\in F''_{k}$. Since $\gamma_{\ell-1}\in E'_{k}$, we have $\Phi_{k}(\gamma_{\ell-1})=\frac{f_{2k+3}}{2}$ and $\Phi_{k}(\gamma_{\ell-1}-f_{2k})=\frac{f_{2k+3}}{2}-f_{2k}=\frac{f_{2k+1}}{2}$. So $(\gamma_{\ell-1}-f_{2k})\in F_{k}$. Suppose on the contrary that $(\gamma_{\ell-1}-f_{2k})\in F'_{k}$. Then $\Phi_{k+1}(\gamma_{\ell-1}-f_{2k})=\frac{f_{2k+1}}{2}$ and $\Phi_{k+1}(\gamma_{\ell-1})=\frac{f_{2k+3}}{2}$. This implies $\Phi_{k}(\gamma_{\ell-1}+f_{2k+2})=\frac{f_{2k+1}}{2}$ and $(\gamma_{\ell-1}+f_{2k+2})\notin E'_k$. Note that in this case $\gamma_{\ell}=\gamma_{\ell-1}+f_{2k+2}$. We conclude that $\gamma_{\ell}\notin E'_{k}$ which is a contradiction. Hence, $(\gamma_{\ell-1}-f_{2k})\in F''_{k}$. The result follows.

	Case 2: $\gamma_{\ell}-\gamma_{\ell-1}=f_{2k+3}$. We assert that, in this case, $\gamma_{\ell-1}-f_{2k}\in F'_{k}$. Since $\gamma_{\ell-1}\in E'_{k}$, we have $\Phi_{k}(\gamma_{\ell-1})=\frac{f_{2k+3}}{2}$. Consequently, $\Phi_{k}(\gamma_{\ell-1}-f_{2k})=\frac{f_{2k+1}}{2}$ and $(\gamma_{\ell-1}-f_{2k})\in F_{k}$. Suppose $(\gamma_{\ell-1}-f_{2k})\in F''_{k}$. Then $\Phi_{k}(\gamma_{\ell-1})=\frac{f_{2k+3}}{2}$ and $\Phi_{k+1}(\gamma_{\ell-1})\neq \frac{f_{2k+3}}{2}$.  It follows that $\Phi_{k}(\gamma_{\ell-1}+f_{2k+3})=\frac{f_{2k+3}}{2}+f_{2k}$. Since $\gamma_{\ell-1}+f_{2k+3}=\gamma_{\ell}$, we obtain that $\gamma_{\ell}\notin E'_{k}$ which is a contradiction. Now we have $\gamma_{\ell-1}-f_{2k}\in F'_{k}$. This yields that $\Phi_{k+1}(\gamma_{\ell-1})=\frac{f_{2k+3}}{2}$. Observing that $\Phi_{k+1}(\gamma_{\ell}-f_{2k})=\Phi_{k+1}(\gamma_{\ell-1}+f_{2k+2})=\frac{f_{2k+5}}{2}$, we see $\gamma_{\ell}-f_{2k}\in E'_{k+1}$. So $(m,n)\in U_k$.
\end{proof}

\section{Relations of Hankel determinants}\label{sec:4}

In this section, we use the Theorem \ref{lem3} to show the determinant value inside $U_k$, $V_k$, $T_k$ is 0. For some integer $k\ge 0$, we prove the relationship between the determinant value of the boundary of $U_k$, $V_k$, $T_k$. We assert that as long as we know one value of $U_k$($V_k$ or $T_k$), we can know all its values.

\subsection{Inside the parallelograms}
The Hankel determinant $H_{m,n}$ vanishes if $(m,n)$ is not on the boundary of any parallelogram $U_{k,i}$, $V_{k,i}$ or $T_{k,i}$.
 \begin{lemma} \label{lem8}
	Let $m\ge 1$ and $n\ge 0$ be two integer.
	\begin{enumerate}
		\item[(i)] If $(m,n)$ is inside $V_{k,i}$ for some $k\ge 0$ and $i\ge 1$, i.e., 
		\[\begin{cases}
			f_{2k}+1\leq n < f_{2k+2}-1,\\
			\beta_i + 1 < n+m \leq \beta_i + f_{2k+1} -1,
		\end{cases}\] then $H_{m,n}=0$.
		\item[(ii)] If $(m,n)$ is inside $T_{k,i}$ for some $k\ge 0$ and $i\ge 1$, i.e., 
		\[\begin{cases}
			f_{2k+1}+1\leq n < f_{2k+2}-1,\\
			\gamma_i - f_{2k} + 1 < n+m \leq \gamma_i -1,
		\end{cases}\] then $H_{m,n}=0$.
		\item[(iii)] If $(m,n)$ is inside $U_{k,i}$ for some $k\ge 0$ and $i\ge 1$, i.e., 
		\[\begin{cases}
			f_{2k}+1\leq n < f_{2k+3}-1,\\
			\alpha_i - f_{2k+2} + 1 < n+m \leq \alpha_i -1,
		\end{cases}\] then $H_{m,n}=0$.
	\end{enumerate}
\end{lemma}

\begin{proof}
	Let $A_{m+i}$ be the $i$-th row of $H_{m,n}$. Then 
	\begin{align*}
	H_{m,n}&=\det \begin{pmatrix}
	s_m & s_{m+1}& \cdots & s_{m+n-1}\\
	s_{m+1} & s_{m+2} &\iddots & \vdots \\
	\vdots & \iddots & \iddots  &\vdots \\
	s_{m+n-1}& \cdots & \cdots & s_{m+2n-2}   
	\end{pmatrix}
	=\det \begin{pmatrix}
	A_m\\
	A_{m+1}\\
	\vdots\\
	A_{m+n-1}
	\end{pmatrix}.
	\end{align*}

	(i) When $m\leq \beta'_i+1$, recall that $\beta'_i=\beta_i-f_{2k}\in F''_{k}$. Since $n\leq f_{2k+2}-2$, by Lemma \ref{lem7}, we have $\Phi_{k}(\beta'_k + j)\neq \frac{f_{2k+1}}{2}$ or $\frac{f_{2k+1}}{2}-1$ for all $1\leq j\leq n$. Then it follows from Theorem \ref{lem3}(i) that 
	\begin{align*}
	A_{\beta'_i+1}&=(s_{\beta'_i+1},\,s_{\beta'_i+2},\,\dots,\, s_{\beta'_i+n})\\
	&=(s_{\beta_i+1},\,s_{\beta_i+2},\,\dots,\, s_{\beta_i+n}) 
	=A_{\beta_i+1},
	\end{align*}
	which gives $H_{m,n}=0$. 
	When $m>\beta'_i+1$, note that $n+m\leq \beta'_k + f_{2k+2}-1$. By Lemma \ref{lem7}, we have $\Phi_{k}(m + j)\neq \frac{f_{2k+1}}{2}$ or $\frac{f_{2k+1}}{2}-1$ for all $1\leq j\leq n$. Then it follows from Theorem \ref{lem3}(i) that 
	\begin{align*}
		A_{m}&=(s_{m},\,s_{m+1},\,\dots,\, s_{m+n-1})\\
		&=(s_{m+f_{2k}},\,s_{m+f_{2k}+1},\,\dots,\, s_{m+f_{2k}+n-1})
		=A_{m+f_{2k}}.
	\end{align*}
	So $H_{m,n}=0$.
	
	(ii) Recall that $\gamma_i\in E'_k$ and by Lemma \ref{lem6}, $\gamma_i$ and $\gamma_i-f_{2k+2}$ are adjacent elements in $E_k$. Let \[r=\begin{cases}
		\gamma_i - f_{2k+2} + 1 - m, & \text{ if } m\le  \gamma_i - f_{2k+2} + 1,\\
		0, & \text{ if } m > \gamma_i - f_{2k+2} + 1.
	\end{cases}\]
	Combining Lemma \ref{lem6} and Theorem \ref{lem3}(ii), we have $A_{m+r}=A_{m+r+f_{2k+1}}$ which means $H_{m,n}=0$.

	(iii) Recall that $\alpha_i\in E'_{k+1}$ and by Lemma \ref{lem6}, $\alpha_i$ and $\alpha_i-f_{2k+4}$ are adjacent elements in $E_{k+1}$. When $m\le  \alpha_i - f_{2k+4} + 1$, note that $\alpha_i-f_{2k+4}+n < \alpha_i - f_{2k+2} - 1$. By Theorem \ref{lem3}(ii), we have 
	\begin{align*}
		A_{\alpha_i-f_{2k+4}+1} & = (s_{\alpha_i-f_{2k+4}+1},\,s_{\alpha_i-f_{2k+4}+2},\,\dots,\,s_{\alpha_i-f_{2k+4}+n})\\
		& = (s_{\alpha_i-f_{2k+2}+1},\,s_{\alpha_i-f_{2k+2}+2},\,\dots,\,s_{\alpha_i-f_{2k+2}+n})
		= A_{\alpha_i-f_{2k+2}+1}.
	\end{align*}
	Thus $H_{m,n}=0$. When $m > \alpha_i - f_{2k+4} + 1$, since $n+m-1\leq \alpha_i-2$, by Theorem \ref{lem3}(ii), we obtain that  
	\begin{align*}
		A_{m} & = (s_{m},\,s_{m+1},\,\dots,\,s_{m+n-1})\\
		& = (s_{m+f_{2k+3}},\,s_{m+f_{2k+3}+1},\,\dots,\,s_{m+f_{2k+3}+n-1})
		= A_{m+f_{2k+3}}
	\end{align*}
	which also implies $H_{m,n}=0$.
\end{proof}

\subsection{Determinants on the horizontal edges of the parallelograms}
We first deal with the Hankel determinants $H_{m,n}$ on the horizontal edges with $n=f_{2k}$ and $f_{2k+1}$ where $k\ge 0$. 
\begin{lemma} \label{lem9}
	Let $k\ge 0$ and $i\ge 1$. 
	\begin{enumerate}
		\item[(i)] \emph{(Bottom edge of $V_{k,i}$)} $H_{\beta'_i+r,\, f_{2k}} = H_{\beta'_i+1,\, f_{2k}}$ for all $1\leq r\leq f_{2k+1}$.
		\item[(ii)] \emph{(Bottom edge of $U_{k,i}$)} $H_{\alpha_i-f_{2k+3}+r,\, f_{2k}} = H_{\alpha_i-f_{2k},\, f_{2k}}$ for all $1\leq r\leq f_{2k+2}$.
		\item[(iii)] \emph{(Bottom edge of $T_{k,i}$)} $H_{\gamma_{i}-f_{2k+2}+r,\, f_{2k+1}}=(-1)^{r+1}H_{\gamma_{i}-f_{2k+1},\, f_{2k+1}}$ for all $1\leq r \leq f_{2k}$ with $\gamma_{i}-f_{2k+2}+r\ge 0$.
	\end{enumerate}
\end{lemma}
\begin{proof}
(i) Let $A_j = (s_{\beta'_i+j},\,s_{\beta'_i+j+1},\,\dots,\, s_{\beta'_i+j+f_{2k}-1})$. Then for $1\leq j < f_{2k+1}$, \[H_{\beta'_i+j,\,f_{2k}}=\det\begin{pmatrix}
	A_j\\ A_{j+1}\\ \vdots\\ A_{f_{2k}+j-1}
\end{pmatrix}\quad\text{and}\quad
H_{\beta'_i+j+1,\,f_{2k}}=\det\begin{pmatrix}
	A_{j+1}\\ A_{j+2}\\ \vdots\\ A_{f_{2k}+j}
\end{pmatrix}.\]
Recall that $\beta'_i\in F''_k$. By Lemma \ref{lem7}, since $j+f_{2k}-1\leq f_{2k+2}-2$, we see $\Phi_k(\beta'_i+\ell)\neq \frac{f_{2k+1}}{2}$ or $\frac{f_{2k+1}}{2}-1$ for all $1\leq \ell \leq f_{2k+2}-2$. Applying Theorem \ref{lem3}(i), we have 
\begin{align*}
	A_j & = (s_{\beta'_i+j},\,s_{\beta'_i+j+1},\,\dots,\, s_{\beta'_i+j+f_{2k}-1})\\
	& = (s_{\beta'_i+j+f_{2k}},\,s_{\beta'_i+j+1+f_{2k}},\,\dots,\, s_{\beta'_i+j+2f_{2k}-1})
	= A_{f_{2k}+j}.
\end{align*}
Therefore, for $1\leq j < f_{2k+1}$,
\[H_{\beta'_i+j,\,f_{2k}}=\begin{vmatrix}
	A_j\\ A_{j+1}\\ \vdots\\ A_{f_{2k}+j-1}
\end{vmatrix}=
\begin{vmatrix}
	A_{f_{2k}+j}\\ A_{j+1}\\ \vdots\\ A_{f_{2k}+j-1}
\end{vmatrix} = (-1)^{f_{2k}-1}
\begin{vmatrix}
	A_{j+1}\\ A_{j+2}\\ \vdots\\ A_{f_{2k}+j}
\end{vmatrix}=H_{\beta'_i+j+1,\,f_{2k}}\]
where the last equality follows from Lemma \ref{lem4}(i).

(ii) Recall that $\alpha_i\in E'_{k+1}$ and $\Phi_{k+1}(\alpha_i)=\frac{f_{2k+5}}{2}$. Let $y=\alpha_i-f_{2k+3}$. Then $\Phi_{k+1}(y)= \frac{f_{2k+1}}{2}$ and $y\in F'_k$. 
Let $B_j = (s_{y+j},\,s_{y+j+1},\,\dots,\, s_{y+j+f_{2k}-1})$. Then for $1\leq j < f_{2k+2}$, \[H_{y+j,\,f_{2k}}=\det\begin{pmatrix}
	B_j\\ B_{j+1}\\ \vdots\\ B_{f_{2k}+j-1}
\end{pmatrix}\quad\text{and}\quad
H_{y+j+1,\,f_{2k}}=\det\begin{pmatrix}
	B_{j+1}\\ B_{j+2}\\ \vdots\\ B_{f_{2k}+j}
\end{pmatrix}.\]
Since $j+f_{2k}-1\leq f_{2k+3}-2$, by Lemma \ref{lem7} and Theorem \ref{lem3}(i),  
\begin{align*}
	B_j & = (s_{y+j},\,s_{y+j+1},\,\dots,\, s_{y+j+f_{2k}-1})\\
	& = (s_{y+j+f_{2k}},\,s_{y+j+1+f_{2k}},\,\dots,\, s_{y+j+2f_{2k}-1})
	= B_{f_{2k}+j}.
\end{align*}
Therefore, for $1\leq j < f_{2k+2}$,
\[H_{y+j,\,f_{2k}}= (-1)^{f_{2k}-1}H_{y+j+1,\,f_{2k}}=H_{y+j+1,\,f_{2k}}\]
where the last equality follows from Lemma \ref{lem4}(i).

(iii) Recall that $\gamma_i\in E'_{k}$. By Lemma \ref{lem6}, $g:=\gamma_i-f_{2k+2}\in E_{k}$. Write \[A_{g+j}=(s_{g+j},\,s_{g+j+1},\,\dots,\,s_{g+j+f_{2k+1}-1}).\]
For $1\leq r < f_{2k}$, 
\[H_{g+r,\,f_{2k+1}} = \begin{vmatrix}
	A_{g+r}\\ A_{g+r+1}\\ \vdots\\ A_{g+r+f_{2k+1}-1}
\end{vmatrix}\quad\text{and}\quad
H_{g+r+1,\,f_{2k+1}} = \begin{vmatrix}
	A_{g+r+1}\\ A_{g+r+2}\\ \vdots\\ A_{g+r+f_{2k+1}}
\end{vmatrix}.\]
By Theorem \ref{lem3}, $A_{g+r}=A_{g+r+f_{2k+1}}$. Then using Lemma \ref{lem4}, for all $1\leq r < f_{2k}$, \[H_{g+r,\,f_{2k+1}}=(-1)^{f_{2k+1}-1}H_{g+r+1,\,f_{2k+1}}=-H_{g+r+1,\,f_{2k+1}}\]
and $H_{g+r,\,f_{2k+1}}=(-1)^{f_{2k}-r}H_{g+r+f_{2k},\,f_{2k+1}}=(-1)^{1+r}H_{g+r+f_{2k},\,f_{2k+1}}.$
\end{proof}

In fact, for all $i\ge 1$, the Hankel determinants on the bottom of $U_{k,i}$ and $V_{k,i}$ take the same value which depends only on $k$. The following lemma helps us to connect the determinants on the bottom of $U_{k,*}$ and $V_{k,*}$.
\begin{lemma}\label{lem-2k}
	Let $k\ge 0$ and $i\ge 1$. 
	If $\gamma_{i+1}-\gamma_i=f_{2k+3}$, then $H_{\gamma_i+f_{2k}+1,\,f_{2k}}=H_{\gamma_{i+1}-f_{2k}+1,\,f_{2k}}$. If $\gamma_{i+1}-\gamma_i=f_{2k+2}$, then $H_{\gamma_i+1,\,f_{2k}}=H_{\gamma_{i+1}-f_{2k}+1,\,f_{2k}}$.
\end{lemma}
\begin{proof}
	Suppose $\gamma_{i+1}-\gamma_i=f_{2k+3}$.  Then $\Phi_{k}(\gamma_{i}+f_{2k})=\frac{f_{2k+3}}{2}+f_{2k}$. Since $3f_{2k}=f_{2k+2}+f_{2k-1}<f_{2k+3}$, by Theorem \ref{lem3}(ii), we have 
	\begin{align*}
		\begin{pmatrix}
			s_{\gamma_i+f_{2k}+1} & \cdots & s_{\gamma_i+3f_{2k}-1}
		\end{pmatrix} & = \begin{pmatrix}
			s_{\gamma_i+f_{2k}+1+f_{2k+1}} & \cdots & s_{\gamma_i+3f_{2k}-1+f_{2k+1}}\end{pmatrix}\\
			& = \begin{pmatrix}
				s_{\gamma_{i+1}-f_{2k}+1} & \cdots & s_{\gamma_{i+1}+f_{2k}-1}
			\end{pmatrix}.
	\end{align*}
	Therefore 
	\begin{align*}
		H_{\gamma_i+f_{2k}+1,\,f_{2k}}& =\begin{vmatrix}
			s_{\gamma_i+f_{2k}+1} & \cdots & s_{\gamma_i+2f_{2k}}\\
			\vdots & \vdots & \vdots\\
			s_{\gamma_i+2f_{2k}} & \cdots & s_{\gamma_i+3f_{2k}-1}	
		\end{vmatrix}  =\begin{vmatrix}
			s_{\gamma_{i+1}-f_{2k}+1} & \cdots & s_{\gamma_{i+1}}\\
			\vdots & \vdots & \vdots\\
			s_{\gamma_{i+1}} & \cdots & s_{\gamma_{i+1}+f_{2k}-1}	
		\end{vmatrix}\\
		& = H_{\gamma_{i+1}-f_{2k}+1,\,f_{2k}}.
	\end{align*}

	When $\gamma_{i+1}-\gamma_i=f_{2k+2}$, we have $\Phi_{k}(\gamma_i)=\frac{f_{2k+3}}{2}$. By Theorem \ref{lem3}(ii), 
	\begin{align*}
		\begin{pmatrix}
			s_{\gamma_i+1} & \cdots & s_{\gamma_i+2f_{2k}-1}
		\end{pmatrix} & = \begin{pmatrix}
			s_{\gamma_i+1+f_{2k+1}} & \cdots & s_{\gamma_i+2f_{2k}-1+f_{2k+1}}\end{pmatrix}\\
			& = \begin{pmatrix}
				s_{\gamma_{i+1}-f_{2k}+1} & \cdots & s_{\gamma_{i+1}+f_{2k}-1}
			\end{pmatrix}.
	\end{align*}
	So $H_{\gamma_i+1,\,f_{2k}}=H_{\gamma_{i+1}-f_{2k}+1,\,f_{2k}}$.
\end{proof}

Next we give the connection between $T_{k,i}$ and $T_{k,i+1}$.
\begin{lemma} \label{lem-2k+1}
	For all $i\ge 1$,  
	$H_{\gamma_i-f_{2k+1},\,f_{2k+1}}=H_{\gamma_{i+1}-f_{2k+1},\,f_{2k+1}}$.
\end{lemma}
\begin{proof}
	If $\gamma_{i+1}-\gamma_i=f_{2k+3}$, then $\Phi_{k+1}(\gamma_i)=\frac{f_{2k+3}}{2}$ and $\Phi_{k+1}(\gamma_i+f_{2k+1})<\frac{f_{2k+5}}{2}$. By Theorem \ref{lem3}(ii), we have 
	\begin{align*}
		\begin{pmatrix}
			s_{\gamma_{i}-f_{2k+1}} & \cdots & s_{\gamma_{i}+f_{2k+1}-2}
		\end{pmatrix} & = \begin{pmatrix}
			s_{\gamma_{i}-f_{2k+1}+f_{2k+3}} & \cdots & s_{\gamma_{i}+f_{2k+1}-2+f_{2k+3}}
		\end{pmatrix} \\
		& = \begin{pmatrix}
			s_{\gamma_{i+1}-f_{2k+1}} & \cdots & s_{\gamma_{i+1}+f_{2k+1}-2}
		\end{pmatrix}.
	\end{align*} 
	Consequently, $H_{\gamma_i-f_{2k+1},\,f_{2k+1}}=H_{\gamma_{i+1}-f_{2k+1},\,f_{2k+1}}$.

	If $\gamma_{i+1}-\gamma_i=f_{2k+2}$, then  $\Phi_{k+1}(\gamma_i)=\frac{f_{2k+3}}{2}+f_{2k+3}$ or $\frac{f_{2k+3}}{2}+f_{2k+4}$. By Theorem \ref{lem3}(i), we have 
	\begin{align*}
		\begin{pmatrix}
			s_{\gamma_{i}-f_{2k+1}} & \cdots & s_{\gamma_{i}+f_{2k+1}-2}
		\end{pmatrix} & = \begin{pmatrix}
			s_{\gamma_{i}-f_{2k+1}+f_{2k+2}} & \cdots & s_{\gamma_{i}+f_{2k+1}-2+f_{2k+2}}
		\end{pmatrix} \\
		& = \begin{pmatrix}
			s_{\gamma_{i+1}-f_{2k+1}} & \cdots & s_{\gamma_{i+1}+f_{2k+1}-2}
		\end{pmatrix}.
	\end{align*} 
	Consequently, $H_{\gamma_i-f_{2k+1},\,f_{2k+1}}=H_{\gamma_{i+1}-f_{2k+1},\,f_{2k+1}}$.
\end{proof}

According to Lemma \ref{lem-2k} and Lemma \ref{lem-2k+1}, the values of the determinants on the bottom edges of $U_{k,i}$ and $V_{k,i}$ only depends on $k$. We improve Lemma \ref{lem9} to the following proposition.
\begin{proposition} \label{prop-i}
	Let $k\ge 0$. For all $i\ge 1$, 
	\begin{enumerate}
		\item[(i)] \emph{(Bottom edges of $U_{k,i}$ and $V_{k,i}$)} for all $1\leq r\leq f_{2k+1}$ and $1\leq r'\leq f_{2k+2}$, \[H_{\alpha_i-f_{2k+3}+r',\, f_{2k}} = H_{\beta'_i+r,\, f_{2k}} = H_{\alpha_1-f_{2k},\, f_{2k}};\]
		\item[(ii)] \emph{(Bottom edge of $T_{k,i}$)} $H_{\gamma_{i}-f_{2k+2}+r,\, f_{2k+1}}=(-1)^{r+1}H_{\gamma_{1}-f_{2k+1},\, f_{2k+1}}$ for all $1\leq r \leq f_{2k}$ with $\gamma_{i}-f_{2k+2}+r\ge 0$.
	\end{enumerate}
\end{proposition}
\begin{proof}
	Since $\alpha_i-f_{2k+2}\in E'_k$ and $\beta'_i+f_{2k}\in E'_k$, Lemma \ref{lem-2k} shows that the values of two determinants on the bottom edge of two adjacent parallelograms in $\{U_{k,j}\}_{j\ge 1}\cup\{V_{k,j}\}_{j\ge 1}$ are the same. Then Lemma \ref{lem9} implies the result (i). The result (ii) follows from Lemma \ref{lem9}(iii) and Lemma \ref{lem-2k+1}.
\end{proof}

\subsection{On the boundary of $U_{k,i}$}
	\begin{lemma} \label{lem13}
		Let $k\ge 0$ and $i\ge 1$. For all $0\leq r\leq f_{2k+2}-1$ with $\alpha_i-f_{2k+4}+2+r\ge 0$, 
		\begin{enumerate}
			\item[(i)] \emph{(Right edge of $U_{k,i}$)}   $H_{\alpha_i-f_{2k+3}+1+r,f_{2k+3}-1-r}=(-1)^{rk}(-1)^{\frac{r(r-1)}{2}}H_{\alpha_i-f_{2k+3}+1,f_{2k+3}-1}$,
			\item[(ii)] \emph{(Left edge of $U_{k,i}$)} $H_{\alpha_i-f_{2k+4}+2+r,f_{2k+3}-1-r}=(-1)^{rk}(-1)^{\frac{r(r-1)}{2}}H_{\alpha_i-f_{2k+3}+1,f_{2k+3}-1}$, 
			\item[(iii)] \emph{(Upper edge of $U_{k,i}$)} $H_{\alpha_i-f_{2k+4}+2+r,f_{2k+3}-1} = (-1)^r H_{\alpha_i-f_{2k+3}+1,f_{2k+3}-1}$.
		\end{enumerate}
    \end{lemma}
	\begin{proof}
		Write $y=\alpha_i-f_{2k+3}$. Recall that $\alpha_i\in E'_{k+1}$. So $\Phi_{k+1}(y)=\frac{f_{2k+1}}{2}$ and $y\in F'_{k}$.
		
		(i)  For $0\le r<f_{2k+2}$, let $A_{y+r+j}$ be the $j$-th column of $M_{y+1+r,\,f_{2k+3}-1-r}$.
		Applying Lemma \ref{lem7} and Theorem \ref{lem3}(i), we see $s_{y+r+\ell}=s_{y+r+\ell+f_{2k}}$ for $1\le \ell \le f_{2k+3}-r-2$ and $s_{y+f_{2k+3}-1}\neq s_{y+f_{2k+3}-1+f_{2k}}$. Then Proposition \ref{prop:2} and Lemma \ref{lem5} yields $s_{y+f_{2k+3}-1}-s_{y+f_{2k+3}-1+f_{2k}}=(-1)^{k}$. Therefore, 
		\[A_{y+r+1} -A_{y+r+f_{2k}}= \begin{pmatrix}
			s_{y+r+1}\\ s_{y+r+2} \\ \vdots\\ s_{y+f_{2k+3}-2}\\ s_{y+f_{2k+3}-1}
		\end{pmatrix}-\begin{pmatrix}
			s_{y+r+1+f_{2k}}\\ s_{y+r+2+f_{2k}} \\ \vdots\\ s_{y+f_{2k+3}-2+f_{2k}}\\ s_{y+f_{2k+3}-1+f_{2k}}
		\end{pmatrix}= \begin{pmatrix}
			0\\ 0\\ \vdots\\ 0\\ (-1)^{k}
		\end{pmatrix}\]
		and 
		\begin{align}\label{eq:u-8}
			H_{y+1+r,\,f_{2k+3}-1-r} & = \begin{vmatrix}
				A_{y+r+1} & A_{y+r+2} & \dots & A_{y+f_{2k+3}-1}
			\end{vmatrix}\nonumber\\
			& = \begin{vmatrix}
				(A_{y+r+1}-A_{y+r+f_{2k}}) & A_{y+r+2} & \dots & A_{y+f_{2k+3}-1}
			\end{vmatrix}\nonumber\\
			& =\begin{vmatrix}
				\mathbf{0}_{ f_{2k+3}-2-r , 1}& M_{y+r+2,f_{2k+3}-r-2}\\
				(-1)^k& *
				\end{vmatrix}\nonumber\\ 
			&=(-1)^k(-1)^{1+f_{2k+3}-1-r}H_{y+r+2,f_{2k+3}-r-2}\nonumber\\
			& = (-1)^{k+r}H_{y+r+2,f_{2k+3}-r-2} 
		\end{align}
		where in the last equality we apply Lemma \ref{lem4} and $\mathbf{0}_{i,j}$ denotes the $i\times j$ zero matrix. It follows from Eq.~\eqref{eq:u-8} that 
		\begin{align*}
			H_{y+1,\,f_{2k+3}-1} & = (-1)^{k}H_{y+2,\,f_{2k+3}-2} =(-1)^{k}(-1)^{k+1}H_{y+3,\,f_{2k+3}-3} \\
			& = (-1)^{k}(-1)^{k+1}\dots(-1)^{k+r-1}H_{y+1+r,\,f_{2k+3}-1-r}\\
			& = (-1)^{rk}(-1)^{\frac{r(r-1)}{2}}H_{y+1+r,\,f_{2k+3}-1-r}.
		\end{align*}
		
		(ii) 
		Let $B_{y-f_{2k+2}+1+r+j}$ be the $j$-th row of $M_{y-f_{2k+2}+2+r,f_{2k+3}-1-r}$. Combining Lemma \ref{lem7}, Theorem \ref{lem3}(i), Proposition \ref{prop:2} and Lemma \ref{lem5}, a similar argument as above yields 
		\begin{align}\label{eq:u-9}
			H_{y-f_{2k+2}+2+r,f_{2k+3}-1-r} & = \begin{vmatrix}
				B_{y-f_{2k+2}+1+r+1}\\
				\vdots\\
				B_{y}\\
				\vdots\\
				B_{y+f_{2k}}
			\end{vmatrix}=\begin{vmatrix}
				B_{y-f_{2k+2}+1+r+1}\\
				\vdots\\
				B_{y}\\
				\vdots\\
				B_{y+f_{2k}}-B_{y}
			\end{vmatrix}\nonumber\\
			& = \begin{vmatrix}
			* & M_{y-f_{2k+2}+2+(r+1),f_{2k+3}-1-(r+1)}\\
			(-1)^k& \mathbf{0}_{ 1 , f_{2k+3}-2-r}
			\end{vmatrix}\nonumber\\
			& = (-1)^{k+r}H_{y-f_{2k+2}+2+(r+1),f_{2k+3}-1-(r+1)}.
		\end{align}
		Applying Eq.~\eqref{eq:u-9}, we have 
		\begin{align*}
			H_{y-f_{2k+2}+2+r,f_{2k+3}-1-r} & = (-1)^{k+r}(-1)^{k+r+1}\cdots(-1)^{k+f_{2k+2}-2}H_{y+1,\, f_{2k}}\\
			& = (-1)^{\frac{(2k+r+f_{2k+2}-2)(f_{2k+2}-r-1)}{2}}H_{y+1,\, f_{2k}}\\
			& = (-1)^{\frac{(2k+r+f_{2k+2}-2)(f_{2k+2}-r-1)}{2}}H_{y+f_{2k+2},\, f_{2k}}\tag{by Lemma \ref{lem9}(ii)}\\
			& = (-1)^{rk}(-1)^{\frac{r(r-1)}{2}}H_{y+1,\,f_{2k+3}-1}. \tag{by Lemma \ref{lem13}(i)}
		\end{align*}

		(iii) Let $C_j$ be the $j$-th column of $M_{y-f_{2k+2}+2+r,f_{2k+3}-1}$. Then 
		\begin{align*}
			H_{y-f_{2k+2}+2+r,f_{2k+3}-1} & = \det(C_1,C_2,\dots,C_{f_{2k+3}-1})\\
			& = \det(C_1,C_2,\dots,C_{f_{2k+3}-1-r}, C'_{1}, C'_{2},\dots,C'_{r})
		\end{align*}
		where $C'_{p}=C_{f_{2k+3}-1-r+p}-C_{f_{2k+3}-1-r+p-f_{2k}}$ for $1\le p \le r$. According to Lemma \ref{lem7}, Theorem \ref{lem3}(i), we have $s_{y+\ell}=s_{y+\ell+f_{2k}}$ for all $1\leq \ell \leq f_{2k+3}-2$ and $f_{2k+3}+1\leq \ell \leq f_{2k+3}+r-2$. 
		By Proposition \ref{prop:2} and Lemma \ref{lem5}, we obtain that $s_{y+f_{2k+3}-1+f_{2k}}-s_{y+f_{2k+3}-1}=(-1)^{k+1}$ and $s_{y+f_{2k+3}+f_{2k}}-s_{y+f_{2k+3}}=(-1)^{k}$. Thus
		\[(C'_{1}, C'_{2},\dots,C'_{r}) = \begin{pmatrix}
			\mathbf{0}_{ f_{2k+3}-1-r, r}\\ X
		\end{pmatrix}\] 
		where $X$ is the $r\times r$ matrix \[\begin{pmatrix}
			0 & \cdots & 0 & (-1)^{k+1}\\
			\vdots & \iddots & \iddots & (-1)^k\\
			\vdots & \iddots & \iddots & \vdots\\
			(-1)^{k+1} & (-1)^k & \cdots & 0
		\end{pmatrix}.\]
		Expanding by the last $r$ columns, we have 
		\begin{align*}
			H_{y-f_{2k+2}+2+r,f_{2k+3}-1} & = \det\begin{pmatrix}
				M_{y-f_{2k+2}+2+r,f_{2k+3}-1-r} & \mathbf{0}_{f_{2k+3}-1-r, r}\\
				* & X
			\end{pmatrix}\\
			&= (-1)^{(k+1)r}(-1)^{\frac{(r-1)r}{2}}H_{y-f_{2k+2}+2+r,f_{2k+3}-1-r}\\
			& = (-1)^{(k+1)r}(-1)^{\frac{(r-1)r}{2}}(-1)^{rk}(-1)^{\frac{r(r-1)}{2}}H_{y+1,f_{2k+3}-1} \tag{by Lemma \ref{lem13}(ii)}\\
			& = (-1)^r H_{y+1,f_{2k+3}-1}. 
		\end{align*}
\end{proof}
\begin{remark} \label{rem4}
	 From Proposition \ref{prop-i}(i) and Lemma \ref{lem13}, Hankel determinants on the boundary of $U_{k,i}$ can be determined by $H_{\alpha_1-f_{2k+3}+1,f_{2k+3}-1}=H_{\frac{f_{2k+1}}{2}+1,f_{2k+3}-1}$ (the upper right corner of $U_{k,1}$).
\end{remark}

\subsection{On the boundary of $V_{k,i}$}

	\begin{lemma} \label{lem14} 
	Let $k\ge 0$ and $i\ge 1$. For all $0\leq r\leq f_{2k+1}-1$,
	\begin{enumerate}
		\item[(i)] \emph{(Left edge of $V_{k,i}$)} $H_{\beta'_i-f_{2k+1}+2+r,f_{2k+2}-1-r}=(-1)^{rk}(-1)^{\frac{r(r+1)}{2}}H_{\beta'_i-f_{2k+1}+2,f_{2k+2}-1}$,
		\item[(ii)] \emph{(Right edge of $V_{k,i}$)} $H_{\beta'_i+1+r,f_{2k+2}-1-r}=(-1)^{rk}(-1)^{\frac{r(r+1)}{2}}H_{\beta'_i+1,f_{2k+2}-1}$,
		\item[(iii)] \emph{(Upper edge of $V_{k,i}$)} $H_{\beta'_i-f_{2k+1}+2+r,f_{2k+2}-1}= H_{\beta'_i+1,f_{2k+2}-1}$.
	\end{enumerate}
\end{lemma}
\begin{proof}
	(i) Denote by $A_{\beta'_i-f_{2k+1}+1+r+j}$ the $j$-th row of $M_{\beta'_i-f_{2k+1}+2+r,f_{2k+2}-1-r}$. Then 
	\begin{align*}
	H_{\beta'_i-f_{2k+1}+2+r,f_{2k+2}-1-r}& = \det \begin{pmatrix}
	A_{\beta'_i-f_{2k+1}+2+r}\\
	A_{\beta'_i-f_{2k+1}+2+r+1}\\
	\vdots\\
	A_{\beta'_i+f_{2k}}
	\end{pmatrix} = \det \begin{pmatrix}
		A_{\beta'_i-f_{2k+1}+2+r}\\
		A_{\beta'_i-f_{2k+1}+2+r+1}\\
		\vdots\\
		A_{\beta'_i+f_{2k}} - A_{\beta'_i}
		\end{pmatrix}.
	\end{align*}
	From Lemma \ref{lem7}, Theorem \ref{lem3} and Lemma \ref{lem5}, we have
	\[	A_{\beta'_i+f_{2k}}-A_{\beta'_i}=((-1)^k,\, 0 ,\, \cdots ,\, 0).
	\]
	For $0\leq r\leq f_{2k+1}-1$,
	\begin{align*}
	H_{\beta'_i-f_{2k+1}+2+r,f_{2k+2}-1-r}&= \begin{pmatrix}
	* & M_{\beta'_i-f_{2k+1}+2+(r+1),f_{2k+2}-1-(r+1)}\\
	(-1)^k& \mathbf{0}_{1 , f_{2k+2}-2-r}
	\end{pmatrix}\\
	& = (-1)^k(-1)^{1+f_{2k+2}-1-r}H_{\beta'_i-f_{2k+1}+2+(r+1),f_{2k+2}-1-(r+1)}\\
	& = (-1)^{k+1+r}H_{\beta'_i-f_{2k+1}+2+(r+1),f_{2k+2}-1-(r+1)}. \tag{by Lemma \ref{lem4}(i)}
	\end{align*}
	Thus	
	\begin{align*}
	H_{\beta'_i-f_{2k+1}+2+r,f_{2k+2}-1-r}&= (-1)^{k+1+r-1}(-1)^{k+1+r-2}\cdots(-1)^{k+1}H_{\beta'_i+2-f_{2k+1},f_{2k+2}-1}\\
	&=(-1)^{r(k+1)}(-1)^{\frac{r(r-1)}{2}}H_{\beta'_i+2-f_{2k+1},f_{2k+2}-1}.
	\end{align*}

	(ii)
	Let $B_{\beta'_i+r+j}$ be the $j$-th column of $M_{\beta'_i+1+r,f_{2k+2}-1-r}$. 
	\begin{align*}
	H_{\beta'_i+1+r,f_{2k+2}-1-r}&=\det\begin{pmatrix}
	s_{\beta'_i+1+r} & s_{\beta'_i+2+r} & \cdots & s_{\beta'_i+f_{2k+2}-1}\\
	\vdots & \vdots & \iddots &  \vdots\\
	s_{\beta'_i+f_{2k+2}-1} & s_{\beta'_i+f_{2k+2}} & \cdots & s_{\beta'_i+2f_{2k+2}-3-r}
	\end{pmatrix}\\
	&=\det\begin{pmatrix}
	B_{\beta'_i+1+r} & B_{\beta'_i+2+r} & \cdots & B_{\beta'_i+f_{2k+2}-1}
	\end{pmatrix}\\
	&=\det\begin{pmatrix}
	B_{\beta'_i+1+r}-B_{\beta'_i+1+r+f_{2k}} & B_{\beta'_i+2+r} & \cdots & B_{\beta'_i+f_{2k+2}-1}
	\end{pmatrix}.
	\end{align*}
	Recall that $\beta'_i\in F''_k$. By Lemma \ref{lem7}, $\beta'_i$ and $\beta'_i+f_{2k+2}$ are adjacent elements in $F_k$.  It follows from Theorem \ref{lem3}(i) and Lemma \ref{lem5} that 
	\begin{align*}
	H_{\beta'_i+1+r,f_{2k+2}-1-r}& =\det\begin{pmatrix}
	\mathbf{0}_{f_{2k+2}-2-r , 1}& M_{\beta'_i+1+(r+1),f_{2k+2}-1-(r+1)}\\
	(-1)^k & *
	\end{pmatrix}\\
	&=(-1)^k(-1)^{1+f_{2k+2}-1-r}H_{\beta'_i+1+(r+1),f_{2k+2}-1-(r+1)}\\
	& = (-1)^{k+1+r}H_{\beta'_i+1+(r+1),f_{2k+2}-1-(r+1)}.\tag{by Lemma \ref{lem4}(i)}
	\end{align*}
	Hence  $H_{\beta'_i+1+r,f_{2k+2}-1-r}=(-1)^{r(k+1)}(-1)^{\frac{r(r-1)}{2}}H_{\beta'_i+1,f_{2k+2}-1}$.

	(iii)
	Let $C_{\beta'_i-f_{2k+1}+1+r+j}$ be the $j$-th column of $M_{\beta'_i-f_{2k+1}+2+r,f_{2k+2}-1}$. Then 
	\begin{align*}
	H_{\beta'_i-f_{2k+1}+2+r,f_{2k+2}-1}
	& = \det \begin{pmatrix}
	C_{\beta'_i-f_{2k+1}+2+r} & C_{\beta'_i-f_{2k+1}+2+r+1} & \cdots & C_{\beta'_i+f_{2k}+r}
	\end{pmatrix}\\
	& = \det \begin{pmatrix}
		C_{\beta'_i-f_{2k+1}+2+r} & \cdots &  C_{\beta'_i+f_{2k}} & C'_{1}& \cdots & C'_{r}
		\end{pmatrix}
	\end{align*}
	where $C'_{p}=C_{\beta'_i+f_{2k}+p}-C_{\beta'_i+p}$ for $1\leq p \leq r$. By Lemma \ref{lem7} and Theorem \ref{lem3}(i), we have $s_{\beta'_i+\ell}=s_{\beta'_i+\ell+f_{2k}}$ for $1\leq \ell \leq f_{2k+2}-2$ and $f_{2k+2}+1\leq \ell \leq r+f_{2k+2}-1$. Moreover, by Proposition \ref{prop:2} and Lemma \ref{lem5}, we have $s_{\beta'_i+f_{2k}+f_{2k+2}-1}-s_{\beta'_i+f_{2k+2}-1}=(-1)^{k+1}$ and $s_{\beta'_i+f_{2k}+f_{2k+2}}-s_{\beta'_i+f_{2k+2}}=(-1)^{k}$. Thus
	\begin{align*}
	\begin{pmatrix}
		C'_{1}& \cdots & C'_{r}
	\end{pmatrix}
	& = \begin{pmatrix}
		\mathbf{0}_{f_{2k+2}-1-r,\, r}\\ X
	\end{pmatrix}
	\end{align*}
	where $X$ is the $r\times r$ matrix 
	\[\begin{pmatrix}
		0 & \cdots & 0 & (-1)^{k+1}\\
		\vdots & \iddots & \iddots & (-1)^k\\
		\vdots & \iddots & \iddots & \vdots\\
		(-1)^{k+1} & (-1)^k & \cdots & 0
		\end{pmatrix}.\]
	Now expanding $H_{\beta'_i-f_{2k+1}+2+r,f_{2k+2}-1}$ by its last $r$ columns, we obtain that for $0\leq r \leq f_{2k+1}-1$,
	\begin{align*}
	H_{\beta'_i-f_{2k+1}+2+r,f_{2k+2}-1}& =\det \begin{pmatrix}
	M_{\beta'_i-f_{2k+1}+2+r,f_{2k+2}-1-r}& \mathbf{0}_{f_{2k+2}-1-r,\, r}\\
	* & X
	\end{pmatrix} \\
	&=(-1)^{(k+1)r}(-1)^{\frac{(r-1)r}{2}}H_{\beta'_i-f_{2k+1}+2+r,f_{2k+2}-1-r}\\
	& = H_{\beta'_i-f_{2k+1}+2,f_{2k+2}-1}. \tag{by Lemma \ref{lem14}(i)}
	\end{align*}
\end{proof}
\begin{remark} \label{rem5}
	 From Proposition \ref{prop-i}(i) and Lemma \ref{lem14}, Hankel determinants on the boundary of $V_{k,i}$ can be determined by $H_{\frac{f_{2k+1}}{2}+f_{2k+3}+1, f_{2k}}$ (the lower left corner of $V_{k,1}$).
\end{remark}
\subsection{On the boundary of $T_{k,i}$}
\begin{lemma} \label{lem15}
	Let $k\ge 0$ and $i\ge 1$. 
\begin{enumerate}
	\item[(i)] \emph{(Left edge of $T_{k,i}$)} For all $0\le r\le f_{2k}-1$ with $\gamma_i-f_{2k+3}+2+r\ge 0$, 
	\[H_{\gamma_i-f_{2k+3}+2+r,\, f_{2k+2}-1-r}=(-1)^{rk}(-1)^{\frac{r(r-1)}{2}}H_{\gamma_i-f_{2k+3}+2,\, f_{2k+2}-1}.\]
	\item[(ii)] \emph{(Right edge of $T_{k,i}$)} For all $0\le r\le f_{2k}-1$ with $\gamma_i-f_{2k+2}+1+r\ge 0$, 
	\[H_{\gamma_i-f_{2k+2}+1+r,\, f_{2k+2}-1-r}=(-1)^{rk}(-1)^{\frac{r(r-1)}{2}}H_{\gamma_i-f_{2k+2}+1,\, f_{2k+2}-1}.\]
	\item[(iii)] \emph{(Upper edge of $T_{k,i}$)} For all $0\le r\le f_{2k}-1$ with $\gamma_i-f_{2k+3}+2+r\ge 0$, 
	\[H_{\gamma_i-f_{2k+3}+2+r,\, f_{2k+2}-1}= H_{\gamma_i-f_{2k+3}+2,\, f_{2k+2}-1}.\]
\end{enumerate}	
\end{lemma}

\begin{proof}
	To shorten the notation, write $x=\gamma_i-f_{2k+3}+2$ and $x'=\gamma_i-f_{2k+2}+1$.

	(i) Let $\max\{0,-x\}\leq \ell \leq f_{2k}-1$ and let $A_{j}$ be the $j$th row of $H_{x+\ell,f_{2k+2}-1-\ell}$.  By Theorem \ref{lem3}(ii) and Lemma \ref{lem5}, we see 
	\[A_{f_{2k+2}-1-\ell}-A_{f_{2k}-1-\ell}=\begin{pmatrix}
		(-1)^{k+1} & 0 & \dots & 0
	\end{pmatrix}.\]
	Then for $\max\{0,-x\}\leq \ell\leq f_{2k}-1$, 
	\begin{align*}
		H_{x+\ell, f_{2k+2}-1-\ell} & = \det\begin{pmatrix}
			A_1\\ A_2\\ \vdots\\ A_{f_{2k+2}-2-\ell}\\ A_{f_{2k+2}-1-\ell}
		\end{pmatrix} = \det\begin{pmatrix}
			A_1\\ A_2\\ \vdots\\ A_{f_{2k+2}-2-\ell}\\ A_{f_{2k+2}-1-\ell}-A_{f_{2k}-1-\ell}
		\end{pmatrix}\\
		& = \begin{vmatrix}
		* & M_{x+(\ell+1),f_{2k+2}-1-(\ell+1)}\\
		(-1)^{k+1}& \mathbf{0}_{1, f_{2k+2}-2-\ell}
		\end{vmatrix}\\
		& = (-1)^{k+1}(-1)^{1+f_{2k+2}-1-\ell}H_{x+(\ell+1),f_{2k+2}-1-(\ell+1)}\\
		& = (-1)^{k+\ell}H_{x+(\ell+1),f_{2k+2}-1-(\ell+1)}. \tag{by Lemma \ref{lem4}(i)}
	\end{align*}
	Applying the above equality $r$ times, one has
	\begin{align*}
		H_{x+r,f_{2k+2}-1-r}&=(-1)^{k+r-1}(-1)^{k+r-2}\cdots(-1)^{k}H_{x,f_{2k+2}-1}\\
		&=(-1)^{rk}(-1)^{\frac{r(r-1)}{2}}H_{x,f_{2k+2}-1}.
	\end{align*}

	(ii) Let $\max\{0,-x'\}\leq \ell \leq f_{2k}-1$ and let $B_j$ be the $j$th column of $H_{x'+\ell, f_{2k+2}-1-\ell}$. By Theorem \ref{lem3}(ii) and Lemma \ref{lem5}, we see 
	\[B_1-B_{1+f_{2k+1}}= \begin{pmatrix}
		0\\  \vdots\\  0 \\ (-1)^{k+1}
	\end{pmatrix}.\]
	Therefore, for $\max\{0,-x'\}\leq \ell \leq f_{2k}-1$,
	\begin{align*}
		H_{x'+\ell, f_{2k+2}-1-\ell} & = \det\begin{pmatrix}
			B_1 & B_2 & \cdots & B_{f_{2k+2}-1-\ell}
		\end{pmatrix}\\
		& = \det\begin{pmatrix}
			B_1-B_{1+f_{2k+1}} & B_2 & \cdots & B_{f_{2k+2}-1-\ell}
		\end{pmatrix}\\
		&=\det \begin{pmatrix}
		\mathbf{0}_{ f_{2k+2}-2-\ell, 1} & M_{x'+(\ell+1),f_{2k+2}-1-(\ell+1)}\\
		(-1)^{k+1}& *\\
		\end{pmatrix}\\
		 &=(-1)^{k+1}(-1)^{1+f_{2k+2}-1-\ell}H_{x'+(\ell+1),f_{2k+2}-1-(\ell+1)}\\
		& = (-1)^{k+\ell}H_{x'+(\ell+1),f_{2k+2}-1-(\ell+1)}. \tag{by Lemma \ref{lem4}(i)}
	\end{align*}
	Applying the above equality $r$ times, one has 
	\begin{align*}
	H_{x'+r,f_{2k+2}-1-r}& =(-1)^{k+r-1}(-1)^{k+r-2}\cdots(-1)^{k}H_{x',f_{2k+2}-1}\\
	&=(-1)^{rk}(-1)^{\frac{r(r-1)}{2}}H_{x',f_{2k+2}-1}.
	\end{align*}

	(iii)
	Let $\max\{0,-x\}\leq r \leq f_{2k}-1$ and let $C_j$ be the $j$th column of $M_{x+r, f_{2k+2}-1}$. Then 
	\begin{align*}
		H_{x+r,f_{2k+2}-1} & = \det\begin{pmatrix}
			C_1 & C_2 & \cdots & C_{f_{2k+2}-1}
		\end{pmatrix}\\
		& = \det\begin{pmatrix}
			C_1 & C_2 & \cdots & C_{f_{2k+2}-r-1} & C'_1 & \cdots & C'_r
		\end{pmatrix}
	\end{align*}
	where $C'_p=C_{f_{2k+2}-r-1+p}-C_{f_{2k}-r-1+p}$ for $1\le p \le r$. Note that 
	\[\begin{pmatrix}
		C'_1 & \cdots & C'_r
	\end{pmatrix} = \begin{pmatrix}
		s_{\gamma_i-f_{2k}+1} & \cdots & s_{\gamma_i-f_{2k}+r}\\
		\vdots & \ddots & \vdots\\
		s_{\gamma_i+f_{2k+1}-1} & \cdots & s_{\gamma_i+f_{2k+1}+r-2}
	\end{pmatrix}-
	\begin{pmatrix}
		s_{\gamma_i-f_{2k+2}+1} & \cdots & s_{\gamma_i-f_{2k+2}+r}\\
		\vdots & \ddots & \vdots\\
		s_{\gamma_i-1} & \cdots & s_{\gamma_i+r-2}
	\end{pmatrix}.\]
	By Lemma \ref{lem6} and Theorem \ref{lem3}(ii), for $1\leq q\leq f_{2k+2}-2$ and $f_{2k+2}+1\leq q \leq f_{2k+2}+r-2$, \[s_{\gamma_i-f_{2k}+q}=s_{\gamma_i-f_{2k+2}+q}.\] Moreover, by Lemma \ref{lem5}, $s_{\gamma_i+f_{2k+1}-1}-s_{\gamma_i-1}=(-1)^{k}$ and $s_{\gamma_i+f_{2k+1}}-s_{\gamma_i}=(-1)^{k+1}$. Then 
	\[\begin{pmatrix}
		C'_1 & \cdots & C'_r
	\end{pmatrix} = 
	\begin{pmatrix}
	\mathbf{0}_{f_{2k+2}-1-r,\, r}\\ X
	\end{pmatrix}\]
	where $X$ is the $r\times r$ matrix 
	\[\begin{pmatrix}
		0 & \cdots & 0 & (-1)^{k}\\
		\vdots & \iddots & \iddots & (-1)^{k+1}\\
		\vdots & \iddots & \iddots & \vdots\\
		(-1)^{k} & (-1)^{k+1} & \cdots & 0
	\end{pmatrix}.\]
	Therefore, 
	\begin{align*}
		H_{x+r,f_{2k+2}-1} &= \det \begin{pmatrix}
		M_{x+r,f_{2k+2}-1-r}& \mathbf{0}_{f_{2k+2}-1-r,\, r}\\
		 * & X
		\end{pmatrix}\\
		& = (-1)^{kr}(-1)^{\frac{(r-1)r}{2}}H_{x+r,f_{2k+2}-1-r}\\
		& = H_{x,f_{2k+2}-1}. \tag{by Lemma \ref{lem15}(i)}
		\end{align*}
\end{proof}
\begin{remark} \label{rem6}
	 From Proposition \ref{prop-i}(ii) and Lemma \ref{lem15}, Hankel determinants on the boundary of $T_{k,i}$ can be determined by $H_{\frac{f_{2k+3}}{2}+f_{2k}+1,f_{2k+1}}$(the lower left corner of $T_{k,2}$).
\end{remark}

\section{Evaluating the Hankel determinants}\label{sec:5}

In section \ref{sec:4}, we show that for any $k\ge 0$, to know all the determinants on the boundary of $U_{k,i}$ (resp. $V_{k,i}$, $T_{k,i}$) for all $i$, it is enough to know the value of one determinant on the boundary $U_{k,i}$ (resp. $V_{k,i}$ or $T_{k,i}$) for some $i$. In this section, for certain $i$, we shall give the expression of a determinant on  the boundary $U_{k,i}$ (resp. $V_{k,i}$ or $T_{k,i}$) for all $k$.

The next result allows us to determine the determinant on the lower left corner of $U_{k,i}$ by using the determinants on the boundary of $U_{k-1, *}$ and $T_{k-1, *}$. 
\begin{lemma} \label{lem16}
	\emph{(Lower left corner of $U_{k,i}$)} For all $k\ge 1$ and $i\ge 1$,  
	\[H_{\alpha_i-f_{2k+3}+1,\,f_{2k}}=(-1)^{k}(H_{\alpha_i-f_{2k+3}+2,\,f_{2k}-1}-H_{\alpha_i-f_{2k+3}+1,\,f_{2k}-1}).\]
\end{lemma}
\begin{proof}
	Let $y=\alpha_i-f_{2k+3}$ and let $A_j$ be the $j$th column of $H_{y+1,f_{2k}}$. Then 
\begin{align*}
	H_{y+1,f_{2k}} & =\begin{vmatrix}
		s_{y+1} & s_{y+2} & \cdots & s_{y+f_{2k}}\\
		\vdots & \vdots & \iddots & \vdots\\
		s_{y+f_{2k}} & s_{y+f_{2k}+1} & \cdots & s_{y+2f_{2k}-1}
	\end{vmatrix}
	 = \det \begin{pmatrix}
		A_1 & A_2 & \cdots & A_{f_{2k}-1}
	\end{pmatrix}.
\end{align*}

Recall that $\alpha_i\in E'_{k+1}$. Then $\Phi_{k+1}(y)=\frac{f_{2k+1}}{2}$ and $\Phi_{k-1}(y+f_{2k})=\frac{f_{2k-1}}{2}$. This implies $y\in F'_k$ and $y+f_{2k}\in F_{k-1}$. By Lemma \ref{lem7} and Theorem \ref{lem3}(i), the fact $y+f_{2k}\in F_{k-1}$ yields that $s_{y+\ell}=s_{y+f_{2k-2}+\ell}$ for $1\leq \ell \leq f_{2k}-2$. By Lemma \ref{lem5}, $s_{y+f_{2k}-1}-s_{y+f_{2k}+f_{2k-2}-1}=(-1)^{k+1}$ and $s_{y+f_{2k}}-s_{y+f_{2k}+f_{2k-2}}=(-1)^{k}$. So 
	\begin{align}
		H_{y+1,f_{2k}} & = \det \begin{pmatrix}
			A_1-A_{1+f_{2k-2}} & A_2 & \cdots & A_{f_{2k}-1}
		\end{pmatrix}\nonumber\\
		& =\begin{vmatrix}
			0 & s_{y+2} & \cdots & s_{y+f_{2k}}\\
			\vdots & \vdots & \iddots & \vdots\\
			0 & s_{y+f_{2k}-1} & \cdots & s_{y+2f_{2k}-3}\\ 
			(-1)^{k+1} & s_{y+f_{2k}} & \cdots & s_{y+2f_{2k}-2}\\ 
			(-1)^{k} & s_{y+f_{2k}+1} & \cdots & s_{y+2f_{2k}-1}
		\end{vmatrix}\nonumber\\
		& = (-1)^{k}(-1)^{1+f_{2k}} H_{y+2, f_{2k}-1} + (-1)^{k+1}(-1)^{f_{2k}}X 
		\label{eq:lem17-1}
	\end{align}
	where \[X=\begin{vmatrix}
		 s_{y+2} & \cdots & s_{y+f_{2k}}\\
		 \vdots & \iddots & \vdots\\
		 s_{y+f_{2k}-1} & \cdots & s_{y+2f_{2k}-3}\\  
		 s_{y+f_{2k}+1} & \cdots & s_{y+2f_{2k}-1}
	\end{vmatrix}=(-1)^{f_{2k}-2}\begin{vmatrix}
		s_{y+f_{2k}+1} & \cdots & s_{y+2f_{2k}-1}\\
		s_{y+2} & \cdots & s_{y+f_{2k}}\\
		\vdots & \iddots & \vdots\\
		s_{y+f_{2k}-1} & \cdots & s_{y+2f_{2k}-3}
   \end{vmatrix}.\]
	Since $y+f_{2k}\in F'_{k}$, by Lemma \ref{lem7} and Theorem \ref{lem3}(i),  we see \[\begin{pmatrix}
		s_{y+f_{2k}+1} & \cdots & s_{y+2f_{2k}-1}
	\end{pmatrix}=\begin{pmatrix}
		s_{y+1} & \cdots & s_{y+f_{2k}-1}
	\end{pmatrix}\]
	and $X=(-1)^{f_{2k}-2}H_{y+1,f_{2k}-1}$. Then the result follows from Eq.~\eqref{eq:lem17-1} and Lemma \ref{lem4}.
\end{proof}

Now we show how to obtain the determinant on the lower left corner of $T_{k,i}$ by using determinants on the boundary of $U_{k,i-1}$ and $U_{k-1,i+1}$.
\begin{lemma} \label{lem17}
	\emph{(Lower left corner of $T_{k,2})$} For all $k\ge 1$ and $i\ge 2$, 
	\[H_{\gamma_i-f_{2k+2}+1,\,f_{2k+1}}=(-1)^{k}(H_{\gamma_i-f_{2k+2}+2,\, f_{2k+1}-1}+H_{\gamma_i-f_{2k+2}+1,\,f_{2k+1}-1}).\]
\end{lemma}
\begin{proof}
	Let $y=\gamma_i-f_{2k+2}$ and let $A_j$ be the $j$th column of $H_{y+1,f_{2k+1}}$. Then 
\begin{align*}
	H_{y+1,f_{2k+1}} & =\begin{vmatrix}
		s_{y+1} & s_{y+2} & \cdots & s_{y+f_{2k+1}}\\
		\vdots & \vdots & \iddots & \vdots\\
		s_{y+f_{2k+1}} & s_{y+f_{2k+1}+1} & \cdots & s_{y+2f_{2k+1}-1}
	\end{vmatrix}
	 = \det \begin{pmatrix}
		A_1 & A_2 & \cdots & A_{f_{2k+1}-1}
	\end{pmatrix}.
\end{align*}

Recall that $\gamma_i\in E'_{k}$. By Lemma \ref{lem6}, $y\in E_k$ and $\Phi_{k}(y+f_{2k+1})=\frac{f_{2k+1}}{2}$. This implies $y+f_{2k+1}\in F_{k}$. By Lemma \ref{lem7} and Theorem \ref{lem3}(i), the fact $y+f_{2k+1}\in F_{k}$ yields that $s_{y+\ell}=s_{y+f_{2k}+\ell}$ for $1\leq \ell \leq f_{2k+1}-2$. By Lemma \ref{lem5}, $s_{y+f_{2k+1}-1}-s_{y+f_{2k+1}+f_{2k}-1}=(-1)^{k}$ and $s_{y+f_{2k+1}}-s_{y+f_{2k+1}+f_{2k}}=(-1)^{k+1}$. So 
	\begin{align}
		H_{y+1,f_{2k+1}} & = \det \begin{pmatrix}
			A_1-A_{1+f_{2k}} & A_2 & \cdots & A_{f_{2k+1}-1}
		\end{pmatrix}\nonumber\\
		& =\begin{vmatrix}
			0 & s_{y+2} & \cdots & s_{y+f_{2k+1}}\\
			\vdots & \vdots & \iddots & \vdots\\
			0 & s_{y+f_{2k+1}-1} & \cdots & s_{y+2f_{2k+1}-3}\\ 
			(-1)^{k} & s_{y+f_{2k+1}} & \cdots & s_{y+2f_{2k+1}-2}\\ 
			(-1)^{k+1} & s_{y+f_{2k+1}+1} & \cdots & s_{y+2f_{2k+1}-1}
		\end{vmatrix}\nonumber\\
		& = (-1)^{k+1}(-1)^{1+f_{2k+1}} H_{y+2, f_{2k+1}-1} + (-1)^{k}(-1)^{f_{2k+1}}X 
		\label{eq:lem18-1}
	\end{align}
	where \[X=\begin{vmatrix}
		 s_{y+2} & \cdots & s_{y+f_{2k+1}}\\
		 \vdots & \iddots & \vdots\\
		 s_{y+f_{2k+1}-1} & \cdots & s_{y+2f_{2k+1}-3}\\  
		 s_{y+f_{2k+1}+1} & \cdots & s_{y+2f_{2k+1}-1}
	\end{vmatrix}=(-1)^{f_{2k+1}-2}\begin{vmatrix}
		s_{y+f_{2k+1}+1} & \cdots & s_{y+2f_{2k+1}-1}\\
		s_{y+2} & \cdots & s_{y+f_{2k+1}}\\
		\vdots & \iddots & \vdots\\
		s_{y+f_{2k+1}-1} & \cdots & s_{y+2f_{2k+1}-3}
   \end{vmatrix}.\]
	Since $y\in E_{k}$, by Lemma \ref{lem6} and Theorem \ref{lem3}(ii),  we see \[\begin{pmatrix}
		s_{y+f_{2k+1}+1} & \cdots & s_{y+2f_{2k+1}-1}
	\end{pmatrix}=\begin{pmatrix}
		s_{y+1} & \cdots & s_{y+f_{2k+1}-1}
	\end{pmatrix}\]
	and $X=(-1)^{f_{2k+1}-2}H_{y+1,f_{2k+1}-1}$. Then the result follows from Eq.~\eqref{eq:lem18-1} and Lemma \ref{lem4}.
\end{proof}

Now We are able to give the exact value of the Hanker determinant on the upper right corner of $U_{k,i}$, and hence we know all the determinants on the boundary of $U_{k,i}$.
\begin{theorem} \label{lem18} 
	\emph{(Upper right corner of $U_{k,1}$)} 
	Let $k\ge 1$. Then \[H_{\frac{f_{2k+1}}{2}+1,\,f_{2k+3}-1}=(-1)^{k+1}\frac{f_{2k+1}}{2}.\]
\end{theorem}
\begin{proof}
We can check directly that the result holds for $k=1,2$. Now suppose $k\ge 3$. Let $h_{k}=H_{\frac{f_{2k+1}}{2}+1,\,f_{2k+3}-1}$. Then 
\begin{align*}
	h_{k} & = (-1)^{(f_{2k+2}-1)k}(-1)^{\frac{(f_{2k+2}-1)(f_{2k+2}-2)}{2}} H_{\frac{f_{2k+5}}{2}-f_{2k},f_{2k}}\tag{by Lemma \ref{lem13}(i)}\\
	& = (-1)^{\frac{f_{2k+2}-1}{2}} H_{\frac{f_{2k+1}}{2}+1,f_{2k}} \tag{by Lemma \ref{lem4} and Lemma \ref{lem9}(ii)}\\
	& = (-1)^{\frac{f_{2k+2}-1}{2}}(-1)^{k}(H_{\frac{f_{2k+1}}{2}+2,\,f_{2k}-1}-H_{\frac{f_{2k+1}}{2}+1,\,f_{2k}-1})\tag{by Lemma \ref{lem16}}.
\end{align*}
Applying Lemma \ref{lem13}(i) for $k-1$\footnote{We need to mention that $\alpha_i$'s depend also on $k$. For example, $\alpha_1^{(k-1)}=\frac{f_{2k+3}}{2}$ and $\alpha_1^{(k)}=\frac{f_{2k+5}}{2}$.} and $r=f_{2k-2}$, 
\begin{align*}
	H_{\frac{f_{2k+1}}{2}+1,\,f_{2k}-1}
	& =  (-1)^{f_{2k-2}(k-1)}(-1)^{\frac{f_{2k-2}(f_{2k-2}-1)}{2}}H_{\frac{f_{2k-1}}{2}+1,\,f_{2k+1}-1}\\
	& = (-1)^{k-1}(-1)^\frac{f_{2k-2}-1}{2}h_{k-1}.\tag{by Lemma \ref{lem4}}
\end{align*}
Applying  Lemma \ref{lem15}(i) for $k-1$ and $r=f_{2k-2}$,
\begin{align*}
	H_{\frac{f_{2k+1}}{2}+2,\,f_{2k}-1}
	& = (-1)^{(f_{2k-2}-1)(k-1)}(-1)^{\frac{(f_{2k-2}-1)(f_{2k-2}-2)}{2}}H_{\frac{f_{2k+1}}{2}+f_{2k-2}+1,\,f_{2k-1}}\\
	& = (-1)^{\frac{f_{2k-2}-1}{2}}H_{\frac{f_{2k+1}}{2}+f_{2k-2}+1,\,f_{2k-1}}\tag{by Lemma \ref{lem4}}\\
	& = (-1)^{\frac{f_{2k-2}-1}{2}}(-1)^{k-1}\left(H_{\frac{f_{2k+1}}{2}+f_{2k-2}+2,\,f_{2k-1}-1}+H_{\frac{f_{2k+1}}{2}+f_{2k-2}+1,\,f_{2k-1}-1}\right)\tag{by Lemma \ref{lem17}}\\
	& = (-1)^{\frac{f_{2k-2}-1}{2}}(-1)^{k-1}\left(h_{k-2}+H_{\frac{f_{2k+1}}{2}+f_{2k-2}+1,\,f_{2k-1}-1}\right) \tag{by Lemma \ref{lem13}(iii)}\\
	& = (-1)^{\frac{f_{2k-2}-1}{2}}(-1)^{k-1}\left(h_{k-2}-h_{k-1}\right). \tag{by Lemma \ref{lem13}(i) and Lemma \ref{lem4}}
\end{align*}
Combing previous equations, we have 
\begin{align*}
	h_{k} & =  (-1)^{\frac{f_{2k+2}-1}{2}}(-1)^{k}(H_{\alpha_1-f_{2k+3}+2,\,f_{2k}-1}-H_{\alpha_1-f_{2k+3}+1,\,f_{2k}-1})\\
	& = (-1)^{\frac{f_{2k+2}-1}{2}}(-1)^{k}(-1)^{\frac{f_{2k-2}-1}{2}}(-1)^{k-1}(h_{k-2}-2h_{k-1})\\
	& = h_{k-2}-2h_{k-1}. \tag{by Lemma \ref{lem4}}
\end{align*}
The initial values are $h_1=2$, $h_2=-5$. The result follows from the recurrence relation of $h_k$ and its initial values.
\end{proof}

\begin{corollary}\label{coro:1}
	\emph{(Lower left corner of $V_{k,1}$)} For all $k\ge 1$, \[H_{\frac{f_{2k+1}}{2}+f_{2k+3}+1,\,f_{2k}}=(-1)^{k+\frac{f_{2k+2}+1}{2}}\frac{f_{2k+1}}{2}.\]
\end{corollary}
\begin{proof}
	By Proposition \ref{prop-i}, 
	\begin{align*}
		H_{\frac{f_{2k+1}}{2}+f_{2k+3}+1,\,f_{2k}} & = H_{\frac{f_{2k+5}}{2}-f_{2k},\,f_{2k}}\\
		& = (-1)^{(f_{2k+2}-1)k}(-1)^{\frac{(f_{2k+2}-1)(f_{2k+2}-2)}{2}}H_{\frac{f_{2k+1}}{2}+1, f_{2k+3}-1}\tag{by Lemma \ref{lem13}(i)}\\
		& = (-1)^{\frac{f_{2k+2}-1}{2}}H_{\frac{f_{2k+1}}{2}+1, f_{2k+3}-1} \tag{by Lemma \ref{lem4}}\\
		& = (-1)^{k+\frac{f_{2k+2}+1}{2}}\frac{f_{2k+1}}{2}. \tag{by Lemma \ref{lem18}}
	\end{align*}
\end{proof}

\begin{corollary} \label{lem19} \emph{(Lower left corner of $T_{k,2}$)} 
	For all $k\ge 1$, $H_{\frac{f_{2k+3}}{2}+f_{2k}+1,f_{2k+1}}=f_{2k}$.
\end{corollary}
\begin{proof}
	From Lemma \ref{lem17}, we have
	\begin{align}
		H_{\frac{f_{2k+3}}{2}+f_{2k}+1,f_{2k+1}}&=(-1)^k(H_{\frac{f_{2k+3}}{2}+f_{2k}+2,f_{2k+1}-1}+H_{\frac{f_{2k+3}}{2}+f_{2k}+1,f_{2k+1}-1}).\label{eq:lem19-1}
	\end{align}
	Note that $H_{\frac{f_{2k+3}}{2}+f_{2k}+2,f_{2k+1}-1}$ is on the upper left corner of $U_{k-1,2}$. By Lemma \ref{lem13}(iii), 
	\[H_{\frac{f_{2k+3}}{2}+f_{2k}+2,f_{2k+1}-1}=H_{\frac{f_{2k-1}}{2}+f_{2k+3}+1,f_{2k+1}-1}.\]
	According to Proposition \ref{prop-i} and Lemma \ref{lem13}, the determinants on the upper left corner of $U_{k-1,1}$ and $U_{k-1,1}$ are equal. Namely, $H_{\frac{f_{2k-1}}{2}+f_{2k+3}+1,f_{2k+1}-1}=H_{\frac{f_{2k-1}}{2}+1,f_{2k+1}-1}$. Therefore, 
	\begin{equation}
		H_{\frac{f_{2k+3}}{2}+f_{2k}+2,f_{2k+1}-1}=H_{\frac{f_{2k-1}}{2}+1,f_{2k+1}-1}.\label{eq:lem19-2}
	\end{equation}
	It follows from Lemma \ref{lem13}(i) and Lemma \ref{lem4} that 
	\begin{align}
		H_{\frac{f_{2k+3}}{2}+f_{2k}+1,,\,f_{2k+1}-1} & =(-1)^{2f_{2k}k}(-1)^{\frac{2f_{2k}(2f_{2k}-1)}{2}}H_{\frac{f_{2k+1}}{2}+1,\,f_{2k+3}-1}\nonumber\\
		& = - H_{\frac{f_{2k+1}}{2}+1,\,f_{2k+3}-1}.\label{eq:lem19-3}
	\end{align}
	Using Eq.~\eqref{eq:lem19-1}, Eq.~\eqref{eq:lem19-2} and Eq.~\eqref{eq:lem19-3}, we have 
	\begin{align*}
		H_{\frac{f_{2k+3}}{2}+f_{2k}+1,f_{2k+1}} & = (-1)^{k}\left(H_{\frac{f_{2k-1}}{2}+1,f_{2k+1}-1}-H_{\frac{f_{2k+1}}{2}+1,\,f_{2k+3}-1}\right)\\
		&=(-1)^k\left((-1)^k\frac{f_{2k-1}}{2}-(-1)^{k+1}\frac{f_{2k+1}}{2}\right)\tag{by Lemma \ref{lem18}}\\
		&=f_{2k}.
	\end{align*}
\end{proof}

\section{Proof of Theorem \ref{thm1}, \ref{thm2} and \ref{thm3}}\label{sec:6}
\begin{proof}[Proof of Theorem \ref{thm1}]
	Suppose $(m,n)\in U_{k,i}$ for some $i$. Then $\alpha_{i}-f_{2k+2}<n+m\leq \alpha_i$ and $m=\alpha_i-f_{2k+2}+1-n+r$ where $0\leq r<f_{2k+2}$. 

	\textbf{Case 1}: $n=f_{2k+3}-1$. Applying Lemma \ref{lem13}(i), Proposition \ref{prop-i} and then Lemma \ref{lem13}(i) again, we have 
	\begin{equation}\label{eq:pf:thm1-1}
		H_{\alpha_i-f_{2k+3}+1,f_{2k+3}-1} = H_{\alpha_1-f_{2k+3}+1,f_{2k+3}-1} 
		 = (-1)^{k+1}\frac{f_{2k+1}}{2}. 
	\end{equation}
		where the last equality follows from Theorem \ref{lem18}. Since  $r=\alpha_i-m-n=\frac{f_{2k+5}}{2}-\Phi_{k+1}(m+n)$, by Lemma \ref{lem13}(iii), 
	\begin{align*}
	H_{m,n}& =H_{\alpha_i-f_{2k+4}+2+r,f_{2k+3}-1}\\
	&= (-1)^r H_{\alpha_i-f_{2k+3}+1,f_{2k+3}-1}\\
	& = (-1)^{\frac{f_{2k+5}}{2}-\Phi_{k+1}(m+n)}H_{\alpha_i-f_{2k+3}+1,f_{2k+3}-1}\\
	& =(-1)^{k+1}(-1)^{\frac{f_{2k+5}}{2}-\Phi_{k+1}(m+n)}\frac{f_{2k+1}}{2}
	\end{align*}
	where the last equality follows from Eq.~\eqref{eq:pf:thm1-1}.

    \textbf{Case 2}: $n=f_{2k}$. By Proposition \ref{prop-i}, we have 
    \begin{align*}
    H_{m,n}& = H_{\alpha_1-f_{2k},\, f_{2k}} \\
           & = (-1)^{\ell k}(-1)^{\frac{\ell(\ell-1)}{2}}H_{\alpha_1-f_{2k+3}+1,f_{2k+3}-1} \tag{by Lemma \ref{lem13}(iii)}\\
           & = (-1)^{k+1}(-1)^{\frac{f_{2k+2}-1}{2}}\cdot \frac{f_{2k+1}}{2}  \tag{by Theorem \ref{lem18}}
	\end{align*}
	where $\ell=f_{2k+3}-1-f_{2k}$ is even by Lemma \ref{lem4}.
    
	\textbf{Case 3}: $f_{2k}<n<f_{2k+3}-1$. If $m+n=\alpha_i$ (or $\alpha_i-f_{2k+2}+1$), then applying Lemma \ref{lem13}(i) (or Lemma \ref{lem13}(ii)) and then Eq.~\eqref{eq:pf:thm1-1}, we have  
	\begin{align*}
		H_{m,n} & = (-1)^{\ell k}(-1)^{\frac{\ell(\ell-1)}{2}}H_{\alpha_i-f_{2k+3}+1,f_{2k+3}-1}\\
		& = (-1)^{\ell k}(-1)^{\frac{\ell(\ell-1)}{2}}(-1)^{k+1}\frac{f_{2k+1}}{2}
	\end{align*}
	where $\ell = f_{2k+3}-1-n$. If $\alpha_i-f_{2k+2}+1<m+n<\alpha_i$, then Lemma \ref{lem8} yields $H_{m,n}=0$.
\end{proof}

\begin{proof}[Proof of Theorem \ref{thm2}]
	Suppose $(m,n)\in V_{k,i}$ for some $i$. Then $\beta_i< n+m\leq \beta_i+f_{2k+1}$.
	
	\textbf{Case 1}: $n=f_{2k+2}-1$. By Lemma \ref{lem14}(iii) \& (i), we have 
	\begin{align}
	H_{m,n} & = H_{\beta'_i-f_{2k+1}+2,\, f_{2k+2}-1} \nonumber\\
			& = (-1)^{(f_{2k+1}-1)k}(-1)^{\frac{(f_{2k+1}-1)f_{2k+1}}{2}}H_{\beta'_i+1,\, f_{2k}}.\label{eq:pf:thm2-1}
	\end{align}
	According to Proposition \ref{prop-i}(i) and Corollary \ref{coro:1}, 
	\begin{equation}
		H_{\beta'_i+1,\, f_{2k}}=H_{\beta'_1+1,\, f_{2k}}=(-1)^{k+\frac{f_{2k+2}+1}{2}}\frac{f_{2k+1}}{2}.\label{eq:pf:thm2-2}
	\end{equation}
	The result follows from Eq.~\eqref{eq:pf:thm2-1} and Eq.~\eqref{eq:pf:thm2-2}.
	
	\textbf{Case 2}: $n=f_{2k}$. By Proposition \ref{prop-i}, $H_{m,n}=H_{\alpha_1-f_{2k},\, f_{2k}}$. Then the result follows from Theorem \ref{thm1}(ii).
	
	\textbf{Case 3}: $f_{2k}<n<f_{2k+2}-1$. If $m+n=\beta_i+1$ (or $\beta_i+f_{2k+1}$), then by Lemma \ref{lem14}(i) (or Lemma \ref{lem14}(ii)), we have 
	\[H_{m,n}=-(-1)^{(f_{2k+2}-n)(k+1)}(-1)^{\frac{(f_{2k+2}-1-n)(f_{2k+2}-2-n)}{2}} \frac{f_{2k+1}}{2}.\]
	If $\beta_i+1<m+n<\beta_i+f_{2k+1}$, then Lemma \ref{lem8} shows $H_{m,n}=0$.
\end{proof}

\begin{proof}[Proof of Theorem \ref{thm3}]
	Suppose $(m,n)\in T_{k,i}$ for some $i$. Then $\gamma_i-f_{2k}< n+m\leq \gamma_i$.
	
	\textbf{Case 1}: $n=f_{2k+2}-1$. By Lemma \ref{lem15}(iii) \& (i),
	\begin{align}
	H_{m,n}& = H_{\gamma_i-f_{2k+3}+2,\, f_{2k+2}-1} \nonumber\\
		   & = (-1)^{(f_{2k}-1)k}(-1)^{\frac{(f_{2k}-1)(f_{2k}-2)}{2}}H_{\gamma_i-f_{2k+2}+1,\, f_{2k+1}}\nonumber\\
		   & = (-1)^{\frac{f_{2k}-1}{2}}H_{\gamma_i-f_{2k+2}+1,\, f_{2k+1}}.\label{eq:pf:thm3-1}
	\end{align}
	where the last equality follows from Lemma \ref{lem4}(i). Using Proposition \ref{prop-i}(ii) and Corollary \ref{lem19}, 
	\begin{equation}
		H_{\gamma_i-f_{2k+2}+1,\, f_{2k+1}}=H_{\gamma_2-f_{2k+2}+1,\, f_{2k+1}}=f_{2k}.\label{eq:pf:thm3-2}
	\end{equation}
	The result follows from Eq.~\eqref{eq:pf:thm3-1} and Eq.~\eqref{eq:pf:thm3-2}.

	\textbf{Case 2}: $n=f_{2k+1}$. Write $m=\gamma_i-f_{2k+2}+r$ with $1\leq r\leq f_{2k}$. By Proposition \ref{prop-i}(ii) and Corollary \ref{lem19}, 
	\begin{align*}
	H_{m,n} &  = (-1)^{r+1}H_{\gamma_{2}-f_{2k+2}+1,\, f_{2k+1}} = (-1)^{r+1}f_{2k}.	
	\end{align*}
	Note that $\gamma_i=m+n+f_{2k}-r$ and $1\leq r\leq f_{2k}$. Consequently, \[\frac{f_{2k+3}}{2}=\Phi_k(\gamma_i)=\Phi_{k}(m+n+f_{2k}-r)=\Phi_{k}(m+n)+f_{2k}-r\]
	which gives $r=\Phi_{k}(m+n)-\frac{f_{2k+1}}{2}$. Then the result follows.
	
	\textbf{Case 3}: $f_{2k}<n<f_{2k+2}-1$. If $m+n=\gamma_i-f_{2k}+1$ (or $\gamma_i$), then by Lemma \ref{lem15}(i) (or Lemma \ref{lem15}(ii)), 
	\[H_{m,n}=(-1)^{(f_{2k+2}-1-n)k}(-1)^{\frac{(f_{2k+2}-1-n)(f_{2k+2}-2-n)}{2}}(-1)^{\frac{f_{2k}-1}{2}}\cdot f_{2k}.\]
	If $\gamma_i-f_{2k}+1<m+n<\gamma_i$, then Lemma \ref{lem8} yields $H_{m,n}=0$.
\end{proof}

\section*{Acknowledgement}
    This work was finished while W. Wu was visiting the Department of Mathematics and Statistics, University of Helsinki. The visit was supported by China Scholarship Council (File No. 201906155024). The research was partially supported by Guangdong Natural Science Foundation (Nos. 2018A030313971, 2018B0303110005).


\begin{thebibliography}{99}
		
	\bibitem{APWW98} J.-P. Allouche, J. Peyri\`{e}re, Z.-X. Wen, Z.-Y. Wen.
		Hankel determinants of the Thue-Morse sequence.
		\emph{Ann. Inst. Fourier}, Grenoble, \textbf{48} (1998), 1–27.

    \bibitem{AS03} J.-P. Allouche, J. Shallit. 
        Automatic sequences: Theory and Applications.
		{\em Cambridge University Press, 2003.}
		
	\bibitem{Bug11} Y. Bugeaud. 
		On the rational approximation to the Thue–Morse–Mahler numbers. 
		\emph{Ann. Inst. Fourier}, \textbf{61} (2011), 2065–2076.

	\bibitem{BHWY16} Y. Bugeaud, G.-N. Han, Z.-Y. Wen, J.-Y. Yao. 
		Hankel Determinants, Pad\'{e} Approximations, and Irrationality Exponents. 
		\emph{Int. Math. Res. Notices} (IMRN), \textbf{5} (2016), 1467–1496.
	 
	\bibitem{Coons13} M. Coons. 
		On the rational approximation of the sum of the reciprocals of the Fermat numbers. 
		\emph{Ramanujan J}. \textbf{30} (1) (2013), 39–65.

	\bibitem{TL1} H. Fu, G.-N. Han. 
		Computer assisted proof for Apwenian sequences.  
		\emph{Proc. ISSAC 2016 Conference}, Waterloo, Ontario, Canada, 2016.
        
	\bibitem{PDS} R. J. Fokkink, C. Kraaikamp, J. Shallit.  
		Hankel matrices for the period-doubling sequence. 
		\emph{Indagationes Mathematicae}, \textbf{28}(1) (2017), 108-119.
		
	\bibitem{GHW20} Y.-J. Guo, G.-N. Han, W. Wu. 
		Criterions for apwenian sequences. 
		\emph{Preprint}, arXiv:2001.10246 (2020), 25 pages.
        
	\bibitem{PF2} Y.-J. Guo, Z.-X. Wen, W. Wu.  
		On the irrationality exponent of the regular paperfolding numbers. 
		\emph{Linear Algebra Appl}. \textbf{446} (2014), 237–264.
		
	\bibitem{Han16} G.-N. Han. 
		Hankel continued fraction and its applications. 
		\emph{Adv. Math.}, \textbf{303} (2016), pp.295–321.
        
	\bibitem{TAMURA} T. Kamae, J. Tamura, Z.-Y. Wen. 
		Hankel determinants for the Fibonacci word and Pad\'{e} approximation. 
		\emph{Acta Arithmetica}, \textbf{89}.2 (1999) 123-161. 

	\bibitem{TAM} J. Tamura. 
		Pad\'e approximation for the inﬁnite words generated by certain substitutions and Hankel determinants. 
		\emph{Number Theory and Its Applications}, Kluwer Academic Publishers. (1990), 309-346.
		
	\bibitem{TW03} B. Tan, Z.-Y. Wen. 
		Invertible substitutions and Sturmian sequences. 
		\emph{European J. Combin.}, \textbf{24}(8) (2003), 983–1002.

	\bibitem{keijo15} K. V\"a\"an\"anen. 
		On rational approximations of certain Mahler functions with a connection to the Thue-Morse sequence. 
		\emph{Int. J. Number Theory}, \textbf{11}(2) (2015), 487–493.                       
\end{thebibliography}
\end{document}